\documentclass[12pt]{article}

\usepackage[T1]{fontenc}

\usepackage[centering, margin={0.8in, 0.5in}, includeheadfoot]{geometry}

\usepackage{amsmath,amssymb,amsthm,amsfonts,mathtools,latexsym,amscd}
\usepackage{enumitem}
\usepackage{epsfig,epstopdf,graphicx,graphics}
\usepackage[hidelinks]{hyperref}
\usepackage[mathscr]{eucal}

\newtheorem{theorem}{Theorem}[section]

\newtheorem{assumption}{Assumption}[section]

\usepackage[nameinlink]{cleveref}
	
\crefname{assumption}{{Assumption}}{Assumptions}
\crefname{lemma}{{Lemma}}{{Lemmas}}
\crefname{theorem}{{Theorem}}{{Theorems}}
\crefname{remark}{{Remark}}{{Remarks}}
\crefname{figure}{{Figure}}{{Figures}}
\crefname{tabl}{{Table}}{{Tables}}
\crefname{cor}{{Corollary}}{{Corollaries}}
\crefname{algorithm}{{Algorithm}}{{Algorithms}}

\theoremstyle{remark}

\newtheorem{remark}{Remark}[section]

\newcommand\scaleQuiverPlots{.4}
\newcommand\scaleParticlePlots{.42}
\newcommand\scaleQuiverPlotsSmall{.3}

\usepackage{float} 
\usepackage{color, xcolor}
\usepackage{cases}
\usepackage{bbm}

\usepackage{algorithmic}
\usepackage[section]{algorithm}

\usepackage{subcaption}
\usepackage{tikz}

\usepackage{verbatim}

\def\cL {{\cal L}}

\def\scrQ{{\mathscr{Q}}}

\newcommand\dive{\mathop{\rm div}}

\graphicspath{{./pictures/}}

\allowdisplaybreaks

\numberwithin{equation}{section}

\input{shortcutsPhD.sty} 

\makeatletter
\renewenvironment{algorithm}
{%
	\bigskip
		\begin{center}
			\refstepcounter{algorithm}
			\hrule height.8pt depth0pt \kern2pt
			\renewcommand{\caption}[2][\relax]{
				{\raggedright\textbf{\fname@algorithm~\thealgorithm} ##2\par}%
				\ifx\relax##1\relax 
				\addcontentsline{loa}{algorithm}{\protect\numberline{\thealgorithm}##2}%
				\else 
				\addcontentsline{loa}{algorithm}{\protect\numberline{\thealgorithm}##1}%
				\fi
				\kern2pt\hrule\kern2pt
			}
	}
	{%
		\kern2pt\hrule\relax
		\bigskip
	\end{center}
}
\makeatother

\begin{document}
\title{\LARGE \bf On the stabilization of a kinetic model by 
feedback-like control fields in a Monte Carlo framework }

\author{
Jan Bartsch{\footnote{\textsc{J. Bartsch}:
\texttt{jan.bartsch@uni-konstanz.de}; Fachbereich f\"ur Mathematik und Statistik,
Universit\"at Konstanz,
Universitätsstraße 10,
78464 Konstanz,
Germany.}}\and
Alfio Borz\`i{\footnote{\textsc{A. Borz\`i}:
\texttt{alfio.borzi@mathematik.uni-wuerzburg.de}; Institut f\"ur Mathematik,
Universit\"at W\"urzburg,
Emil-Fischer-Strasse 30,
97074 W\"urzburg,
Germany.}}
}
\date{\small\today}

\maketitle

\begin{abstract}
	The construction of feedback-like control fields for a kinetic model in phase space
	is investigated. The purpose of these controls is to drive an initial density of
	particles in the phase space to reach a desired cyclic
	trajectory and follow it in a stable way. For this purpose,
	an ensemble optimal control problem governed by the kinetic model
	is formulated in a way that is amenable to a Monte Carlo approach.
	The proposed formulation allows to define a one-shot solution procedure
	consisting in a backward solve of an augmented adjoint kinetic model.
	Results of numerical
	experiments demonstrate the effectiveness of the proposed control strategy.
\end{abstract}

\paragraph*{Keywords:}{
\small Kinetic models in phase space, Keilson-Storer collision term, ensemble optimal control problems, feedback control, Monte Carlo methods.
}


\section{Introduction}

This work is devoted to the formulation and investigation
of a novel optimal control
problem governed by a linear kinetic model with linear collision term and
subject to specular reflection boundary conditions.
Our kinetic model consists of a Liouville-type non-homogeneous
streaming operator and a linear collision term $C[f]$ as follows:
\begin{equation}
	\partial_t f(x,v,t) + \nabla_x \cdot \big( v \,  f(x,v,t) \big)
	+ \nabla_v \cdot \big(F(x,v,t) \, f(x,v,t)	\big)=
	C[f](x,v,t) .
	\label{eq:linear_transport_general}
\end{equation}

We consider an initial- and boundary-value problem
with this model in the phase space $\Omega \times \RR^d$, where
$x \in \Omega \subset \RR^d$ represents the position space coordinate
and $v \in \RR^d$ represents the velocity. We assume that $\Omega$ is bounded and convex with piecewise smooth boundary $\partial \Omega$;
we denote with $n(x)$ the unit outward normal vector to
$\partial \Omega$ at $x \in \partial \Omega$. On the inflow part of this
boundary, we require (partial) specular reflection boundary conditions given by
\begin{equation}
	f(x,v,t) = \alpha \, f(x,v-2\, n(x) \, (n(x)\cdot v),t) ,
	\label{eq:forward_bc}
\end{equation}
where the parameter $\alpha \in [0,1]$ resembles the probability that a particle
colliding with the space boundary is reflected or absorbed into the boundary.
We have an inflow boundary at $(x,v) \in \partial \Omega \times \RR^d$
if it holds
$$
v \in \RR^d_< := \{v \in \RR^d \, | \, v\cdot n(x)<0\} .
$$
Our evolution problem is considered in the time interval $[0,T]$, $T>0$.

We remark that, starting with the pioneering work \cite{BealsProtopopescu1987AbstractTransport},
problem \eqref{eq:linear_transport_general} has been subject of theoretical investigation in, e.g.,
\cite{ChenYang1991LinearTransportSpecularReflection, LatrachLods2009TransportBounceBack} in the homogeneous case where $F=0$, and in \cite{VanderMee1991NonDivFreeForceKinEQ} in the more general non-homogeneous case with $F$ non divergence free.

Our purpose is to design a control field capable of driving an initial density of
particles randomly distributed in the phase space, with
corresponding density denoted with $f_0(x,v)$, to reach a desired cyclic
trajectory on the phase space and follow it in a stable way. The desired
trajectory is defined as the solution to the following differential problem
\begin{equation}
	X'(t)=V(t), \qquad V'(t)=F_0(X(t),V(t)),
	\label{eDiySys}
\end{equation}
where the dynamics $F_0$ with the initial condition $X(0)=X_0$ and $V(0)=V_0$,
$(X_0,V_0) \in \Omega \times \RR^d$ are chosen such that the resulting
periodic trajectory satisfies $(X(t),V(t)) \in \Omega \times \RR^d$, $t \in [0,T]$.
In the experiments, we consider the case where the given force field
$F_0$ corresponds
to Hooke's law. Therefore assuming that all particles are initially concentrated
on $(X_0,V_0)$, in the absence of collision they will follow an elliptic
trajectory on the phase space including this point.

In general, if the choice of $F_0$ does not result in a stable limit cycle
dynamics, starting with a distributed $f_0$ will result in particles
following different trajectories. In this case, an additional control field
must be designed to drive and maintain the particles, subject to collisions,
close to the desired trajectory. For this purpose, we augment $F_0$
with a control field as follows
\begin{align*}
	F(x,v,t;u) = F_0(x,v) + u(x,v,t).
\end{align*}
Notice that the control $u$ resembles, in the field of stochastic
control theory, a so-called Markov control function, in the sense that
a particle being at $(X(t),V(t))$ will be subject to the force $u(X(t),V(t),t)$
that instantaneously acts on the particle to perform the given task.

In our work, the task of driving the particles moving subject to
the force $F_0$, to a beam along the desired trajectory can be
considered as an idealization of problems concerning the control of plasma
currents. In this framework, a time-dependent control $u(x,v,t)$ is
required in the transient ignition phase, whereas for steady-state operation
a stationary feedback law $u(x,v)$ would be required.

We remark that, in both cases, these control fields
can be determined as solutions to kinetic optimal control problems with
ensemble cost functionals. In the realm of optimal control theory,
these functionals have been proposed
and studied in \cite{Brockett1997,Brockett2007,Brockett2012}
and later in \cite{Bartsch2019Theoretical,Bartsch2020OCPKS}. However, they have been well-known in stochastics
and statistical mechanics for a long time as expected value functionals.
In our case, we consider ensemble cost functionals with the following
structure
\begin{align}
	J(f,u) = \int_0^T \int_{\Omega \times \RR^d} \ell (x,t,v,u) \, f(x,v,t)\,dx\, dv\,dt
	+ \int_{\Omega \times \RR^d} \varphi(x,v)f(x,v,T)\,dx\,dv ,
	\label{eq:objective0}
\end{align}
where $\ell$ encodes the purpose of the control and its cost, and
$\varphi$ defines the objective of the control at final time.

We investigate the optimal control problem of minimizing \eqref{eq:objective0}
subject to the differential constraint given
by \eqref{eq:linear_transport_general}, and focus on the corresponding
optimality system. In particular, we analyze the optimality condition
equation for the reduced gradient and derive a condition
for optimality.
In this way, we arrive at the formulation of an augmented
nonlinear adjoint model for computing the optimal control field that consists in a one-shot procedure of solving the adjoint model once backwards in time, i.e. without the need of an iterative optimization scheme.
This backward solution is computed using a novel Monte Carlo approach.

We would like to point out that the investigation of open-loop optimal control problems governed by kinetic models with collision is an emerging topic with only few contributions \cite{AlbiChoiFornasierKalise2017MeanFieldControlHierarchy,Bartsch2021MOCOKI,Bartsch2020OCPKS,CaflischSilantyevYang2021AdjointDSMC,LiWangYang2023MCGradient,YangSilantyevCaflisch2023AdjointDSMC_GeneralCollisionModel}.

This is particularly true for works concerning the solution of these problems in the framework of Monte Carlo methods.
However, our approach  is the first that allows to compute feedback-like controls.

This paper is organized as follows. In the next section, we specify our governing
model and the structure of the ensemble cost functional, and
formulate our optimal control problem.
In Section \ref{sec:FirstOrderOptimalityConditions}, we use the Lagrange
framework to derive the first-order necessary optimality conditions including
an adjoint kinetic model and an optimality-condition equation. Along this
process, we recognize that a sufficient condition for satisfying the
optimality equation leads to a direct relation between the velocity
gradient of the adjoint variable and the control sought. This is the key
step to determine our feedback-like control fields as the solution
of a nonlinear augmented adjoint model. Furthermore, we propose a
way to obtain a stationary feedback law.
Section \ref{sec:MC} is devoted to the
development of a Monte Carlo (MC) framework for simulating our
kinetic and adjoint models in the specific case of a Keilson-Storer collision kernel \cite{KeilsonStorer52}. For this purpose, we further develop our
MC solver \cite{Bartsch2021MOCOKI} that was proposed to solve a
semi-ensemble open-loop optimal
control problem with a space-dependent control.
However, in the present case the control depends on $x$ and $v$
and is time-dependent. The main focus
of this section is the formulation of the adjoint problem as a
kinetic model with additional source- and reaction terms such that
they can be accommodated in a MC setting. The
use of the direct relation between
the control function and the velocity gradient of the adjoint variable,
and the reformulation of the adjoint model, allow
to construct a one-shot method for determining the control sought.
Results of numerical experiments
are presented in Section \ref{sec:NumericalExamples} considering
the trajectory of a harmonic oscillator as the desired orbit.
Correspondingly, we demonstrate that the time-dependent control field
obtained with our method and the corresponding time average are
able to perform the given tasks.
A section of conclusion completes this work.

\section{A kinetic optimal control problem}
\label{sec:FormulationOfTheProblem}
In this section, we formulate our ensemble optimal control problem.
The evolution of the density is governed by \eqref{eq:linear_transport_general},
where the gain-loss collision term $C[f](x,v,t)$ is given by
\begin{equation}
	C[f](x,v,t) = \int_{\RR^d} A(w,v)\,f(x,w,t)\,dw -\sigma(x,v)\, f(x,v,t) ,
	\label{CollisionTerm}
\end{equation}
where the function $\sigma$ represents the collision frequency,
and $A$ is the collision kernel. Both are assumed to be non-negative and bounded.

In order to ease notation, we introduce the free-streaming
operator $L_u$ (in conservative form) given by
\begin{align*}
	L_u = \nabla_x \big( v \, \cdot \big) + \nabla_v \big( F(x,v,t; u) \, \cdot \big)  \, .
\end{align*}
Hence, our controlled kinetic model can be written as follows:
\begin{equation*}
	\partial_t f(x,v,t) + L_u \, f(x,v,t) =  C[f](x,v,t) .
\end{equation*}
Further, we specify an initial density $f_0(x,v) \ge 0$
for all $(x,v) \in \Omega \times \RR^d$ at time $t=0$ such that
$\lim_{|v| \rightarrow \infty} f_0(x,v) \rightarrow 0$ for all $ x \in \Omega$.

Our problem is posed on the phase-space-time cylinder
$\scrQ \coloneqq \Omega \times \RR^d \times (0,T]$
and the inflow and outflow boundaries are defined as follows:
\begin{align}
	\scrQ^- \coloneqq \partial \Omega \times \RR^d_<\times (0,T],
	\qquad\qquad
	\scrQ^+ \coloneqq \partial \Omega \times \RR^d_{>}\times (0,T].
\end{align}

Our kinetic initial- and boundary-value problem is given by
\begin{align}
	\partial_t f(x,v,t) \, + \, L_u \,  f(x,v,t) \,
	&= \, C[f](x,v,t) \,  & \text{ in }& \scrQ \notag \\
	f(x,v,0) \,& = \, f_0(x,v) & \text{ on } &\Omega \times \RR^d
	\label{eq:controlled_Init_bdry_value_problem}\\
	f(x,v,t) &= \alpha \, f(x,v-2 \, n(x) \, (n(x)\cdot v),t) & \text{ on }  &\scrQ^- \notag.
\end{align}

Well-posedness of this evolution problem can be stated subject to the
following assumptions
\begin{assumption} \label{asmpt:Existence_Assumption_VanDerMee}
	\phantom{}
	\begin{enumerate}[label=\alph*)]
		\item $F$ \emph{is a Lipschitz-continuous vector field.}
		\item\emph{ The integral} $\int_0^s \dive F(x,v,t;u)\, dt$ \emph{is essentially bounded on} $\scrQ$.
		\item \emph{The collision frequency }$\sigma$ \emph{is integrable and bounded from below by the divergence of} $F$.
	\end{enumerate}
\end{assumption}
Subject to these conditions, existence and uniqueness of solutions
is stated in the following theorem \cite[Theorem 4]{VanderMee1991NonDivFreeForceKinEQ}:
\begin{theorem}
	\label{thm:Existence_VanDerMee}
	Assume that \cref{asmpt:Existence_Assumption_VanDerMee} holds and suppose $1\leq p < \infty$ and $0 \leq \alpha < 1$.
	Furthermore, assume that  and $f_0 \in L^p(\Omega \times \RR^d)$.
	Then there exists a unique $f \in L^p(\scrQ)$ that solves \eqref{eq:controlled_Init_bdry_value_problem}.
	Furthermore, if the initial condition are nonnegative and the collision operator is a positive operator, then $f$ is nonnegative.
\end{theorem}

By virtue of \cref{thm:Existence_VanDerMee}, we see that choosing a fixed initial condition $f_0  \in L^p(\Omega \times \RR^d)$, for a given control $u \in C^{0,1}(\scrQ)$, the solution to \eqref{eq:controlled_Init_bdry_value_problem} results in a well-defined control-to-state map $G: C^{0,1}(\scrQ) \rightarrow L^p(\scrQ)$, $u \mapsto f=G(u)$.
Continuity and differentiability of this map can be proved based on well-known
techniques; see, \break e.g., \cite{Bartsch2019Theoretical}.

Our ensemble cost functional is given in \eqref{eq:objective0}
with $\ell(x,t,v,u)=\theta(x,v,t) + \frac{\nu}{2}\, |u|^2$. We have
\begin{multline}
	J(f,u) = \int_0^T \int_{\Omega \times \RR^d}
	\big[ \theta (x,v,t) + \frac{\nu}{2}  |u(x,v,t)|^2  \big]\, f(x,v,t)\,dx \, dv \,dt
	\\ + \int_{\Omega \times \RR^d} \varphi(x,v) f(x,v,T)\,dx \, dv .
	\label{eq:objective}
\end{multline}

We make the following assumption.
\begin{assumption}
	\label{aspt:Convex_Mathematical_Potentials}
	\emph{We suppose that} $\theta$ \emph{and} $\varphi$ \emph{are smooth
		functions, bounded from below, and locally strictly convex in a neighbourhood of their global minimum and having monotone radial growth
		with respect to their respective minimum.}
\end{assumption}

In \eqref{eq:objective}, the term $\theta$ encodes the task of the
control to drive the particles along a desired trajectory in phase space denoted
with $z_D(t)=(x_D(t),v_D(t))$. In the realm of control of ordinary
differential models, the choice $\theta(z,t) = | z-z_D(t)|^2 $, $z=(x,v)$, would be standard.
Similarly, the final observation term could be
defined as $\varphi(z) = |z - z_T|^2$, $z_T=(x_T,v_T)$. Notice
that, with these terms, the functional evaluates mean-square errors
of tracking and final observation. However, with these functions,
we cannot guarantee integrability of the terms
$\theta \, f$ and $\varphi\, f(T)$ in the objective functional. On the other hand,
the following choice is appropriate:
\begin{align}
	\theta(z,t) = - \frac{C_\theta}{\sqrt{(2\pi)^{2d}\det(\Sigma_\theta)}} \exp\left(-\frac{1}{2}(z-z_D(t))^T\Sigma^{-1}_\theta(z-z_D(t)) \right),
	\label{eq:Structure_theta}
\end{align}
where $C_\theta>0$ represents a weight of the tracking part
of the cost functional, and $\Sigma_\theta \in \RR^{2d \times 2d}$ has the significance of a co-variance matrix that we assume to be a diagonal matrix.

Similarly, we choose
\begin{align*}
	\varphi(z) = -\frac{C_\varphi}{\sqrt{(2\pi)^{2d}\det(\Sigma_\varphi)}}\exp \left( -\frac{1}{2}(z-z_T)^T\Sigma^{-1}_\varphi(z-z_T) \right), \quad C_\varphi > 0.
\end{align*}
Notice that the choice of $\theta$ and $\varphi$ given above
satisfy the requirements of \cref{aspt:Convex_Mathematical_Potentials}.

Now, we can formulate our ensemble optimal control problem
\begin{equation}
	\begin{split}
		&\min_{{ (u,f)}}  \int_0^T \int_{\Omega \times \RR^d} { \left( \theta (x,v,t) + \frac{\nu}{2}|u(x,v,t)|^2 \right) } \, f(x,v,t)\,dx\,dv\,dt
		+ \int_{\Omega \times \RR^d} \varphi(x,v)f(x,v,T)\,dx\, dv
		\\[0.2cm]
		&\text{s.t.}
		\begin{cases}
			\partial_t f(x,v,t) + L_u \, f(x,v,t) = C[f](x,v,t), \quad & \text{ in } \scrQ \\[0.2cm]
			f(x,v,0) = f_0(x,v) \quad & \text{ on } \Omega \times \RR^d \\[0.2cm]
			f(x,v,t) = \alpha \, f(x,v-2 \, n(x) \, (n(x)\cdot v),t) & \text{ on }  \scrQ^- \; .
		\end{cases}
	\end{split}
	\label{eq:OCP}
\end{equation}

In this problem, the control field is sought in the space of measurable
functions such that integrability of $|u|^2 f$ in $\scrQ$ is guaranteed for any $f\in L^p(\scrQ)$, $1 \leq p < \infty$.
We denote this space with $U$.
Notice that if $f$ is positive, then the integral $\int_{\scrQ} |u|^2f \,dx\,dv\,dt$ gives the square of the norm of $u$ in a weighted $L^2$ space.
By means of the control-to-state map, the optimization
problem \eqref{eq:OCP} can be reformulated as follows:
\begin{equation}
	\min_{u \in U} J_r(u) ,
	\label{minJred}
\end{equation}
where $J_r(u) := J(G(u),u)$ defines the so-called reduced cost functional.
Existence of a minimizer to \eqref{minJred} (i.e. an optimal control
for \eqref{eq:OCP}) can be proved as in \cite{Bartsch2019Theoretical},
subject to some restrictive assumption.
Specifically, we have the following theorem
\begin{theorem}
	Let the assumptions of \cref{thm:Existence_VanDerMee} and  \cref{aspt:Convex_Mathematical_Potentials} be fulfilled.
	Additionally, suppose that
	$
	u \in U \cap  L^\infty(\scrQ) .
	$
	Then the optimal control problem \eqref{minJred} admits at least one solution $u^* \in U\cap L^\infty(\scrQ)$.
\end{theorem}
\begin{proof}
	The Assumptions (A.1)-(A.4) in \cite{Bartsch2019Theoretical} are fulfilled and we can apply \cite[Theorem 4.1]{Bartsch2019Theoretical} in order to conclude the existence of a minimizer.
\end{proof}

\section{Optimality system}
\label{sec:FirstOrderOptimalityConditions}
In this section, we discuss the optimality system characterizing
a solution to \eqref{eq:OCP}.
For this purpose,
Fr\'echet differentiability of $G$ and $J$ are required, which we assume in this work; however, see, e.g.,
\cite{Bartsch2019Theoretical} for related results.

A convenient way to derive the optimality system is to introduce the Lagrange function corresponding to \eqref{eq:OCP} as follows:
\begin{align*}
	&\cL(f,u,q,q_0,q_\sigma) \\[0.2cm]& = J(f,u) + \int_0^T \int_{\Omega \times \RR^d}
	\Big( \partial_t f (z,t) +L_u f(z,t) - C[f](z,t) \Big)\, q(z,t)\,dz\,dt \\[0.2cm]
	&\quad + \int_{\Omega \times \RR^d} \big( f(z,0) - f_0(z) \big)\, q_0(z)\,dz\\[0.2cm]
	& \quad+ \int_0^T \int_{\partial \Omega \times \RR^d_<}
	(f(x,v,t)-\alpha \, f(x,v-2 \, n \, (n\cdot v),t)) \, q_s(x,v,t)\,d\sigma\,dv\,dt ,
\end{align*}
where $z=(x,v)$, $q$, $q_0$ and $q_s$ represent Lagrange multipliers and $d\sigma$ is a surface element. In this framework, the optimality system is obtained by requiring that the Fr\'echet derivatives of
$\cL$ with respect to its arguments are zero.

The resulting optimality system
consists of three parts: 1) the kinetic model;
2) the adjoint kinetic problem; 3) the optimality condition
that corresponds to $\nabla_u J_r(u)=0$, where this gradient is expressed in terms
of the solutions to the problems 1) and 2).

The adjoint kinetic model is given by
\begin{align}
	-\partial_t q(x,v,t)+ L^*_u \, q(x,v,t) &= \SPI{d} A(v,w) \, q(x,w,t)\,dw& \notag \\[0.2cm]
	& \quad- q(x,v,t) \, \sigma(x,v) - \theta(x,v,t) \nonumber\\
	&\quad- \frac{\nu}{2}|u(x,v,t)|^2 & \text{ in } \scrQ,    \label{eq:AdjointEquation} \\[0.2cm]
	q(x,v,T) &= -\varphi (x,v)&\qquad\text{ in } \Omega\times \RR^d, \notag \\[0.2cm]
	&\hspace{-1.2cm} {q(x,v,t) = \alpha \, q(x,v-2 \, n(x) \, (n(x)\cdot v),t)} & \text{ in } \scrQ^+ \; .
	\notag
\end{align}

Notice that in this problem a terminal condition is specified and therefore
\eqref{eq:AdjointEquation} governs the evolution of $q$ backwards in time.
The adjoint free-streaming operator $L^*_u$ is given by
\begin{align*}
	L^*_u \coloneqq  -v\cdot\nabla_x - {F(x,v,t;u)}\cdot\nabla_v .
\end{align*}

The optimality system is completed with the specification of
the optimality condition equation.
Based on the Lagrange function given above, we obtain
\begin{equation}
	\nabla_u J_r(u):=
	f(x,v,t) \Big(\nu \, u(x,v,t) - \partial_u F(x,v,t;u)\nabla_v q(x,v,t)\Big) =0 .
	\label{eq:ReducedGradient}
\end{equation}

However, recall our aim to construct a control field on the entire
phase space. Obviously, such a control would be required if the density
$f$ is everywhere positive, in which case a necessary and sufficient
condition for \eqref{eq:ReducedGradient} to be to satisfied is to set
\begin{equation}
	u(x,v,t) = \frac{1}{\nu} \nabla_v q(x,v,t) ,
	\label{eFB}
\end{equation}
since in our setting $\partial_u F(x,v,t;u)$ is an identity matrix.
This is an essential step in our development
because with $u$ given by \eqref{eFB} and replaced in the
adjoint kinetic model \eqref{eq:AdjointEquation}, we obtain an equation
for the adjoint variable that does not depend on the density of the
particles nor on its initial condition. Therefore also $u$ given by \eqref{eFB}
does not depend on $f$ but solely on the optimization functions
$\theta$ and $\varphi$ that define the control tasks. These are the
characterizing features of a feedback control.

Therefore, our focus is the solution of the following nonlinear
augmented adjoint problem
\begin{equation}
	\begin{split}
		-\partial_t q(x,v,t) + L^*_u \, q(x,v,t) = \SPI{d} A(v,w) \, q(x,w,t)\,dw \qquad\qquad\qquad \quad \\[0.2cm]
		- q(x,v,t) \, \sigma(x,v) - \theta(x,v,t) - \frac{1}{2 \nu}| \nabla_v q(x,v,t)|^2 \qquad &\text{ in } \scrQ,
		\\[0.2cm]
		q(x,v,T) = -\varphi (x,v) \qquad &\text{ in } \Omega\times \RR^d, \\[0.2cm]
		{q(x,v,t) = \alpha \, q(x,v-2 \, n(x) \, (n(x)\cdot v),t)}\qquad  &\text{ in } \scrQ^+ \; .
		\label{eq:AdjointAugmented}
	\end{split}
\end{equation}

Notice that this problem has some similarity with the Hamilton-Jacobi-Bellman
equation arising in the dynamic programming approach \cite{Bellman1957} to compute closed-loop controls for stochastic drift-diffusion models; see, e.g.,  \cite{Fabbri2017,HorowitzDamleBurdick2014LinearHJB} for recent references on this topic.

We remark that $\theta$ defined in \eqref{eq:Structure_theta} has the purpose
to drive a density such that it remains approximately centred at $z_D(t)$
at time $t$. However, our purpose is to drive the ensemble of particles
to reach and maintain a given periodic orbit, which admits different
time parametrizations. Therefore in this case it is more appropriate
to consider in the objective functional a time-independent $\theta$ corresponding to the entire
trajectory. For this purpose, we consider a time-averaged function as follows
\begin{equation}
	\bar{\theta}(x,v)=\frac{1}{T} \, \int_0^T \theta(x,v,t) \, dt .
	\label{eThetaStat}
\end{equation}
Using \eqref{eq:Structure_theta}, we see that $\bar{\theta}$
represents a closed valley with the bottom line corresponding to
the desired orbit.

In the case that a stationary feedback law is required, we propose
to construct it through the time average of \eqref{eFB} as follows
\begin{equation}
	\bar{u}(x,v)=\frac{1}{T} \, \int_0^T u(x,v,t) \, dt .
	\label{eControlStat}
\end{equation}
This approach is motivated by works on
optimal control problems of periodic processes in the field of engineering of chemical plants with cyclic regimes; see, e.g., \cite{SpeyerEvans1984}.

Indeed, we face the problem of solving the completely new differential
problems \eqref{eq:AdjointAugmented}, which is a
challenge with deterministic numerical methods and even more
challenging with Monte Carlo methods
as we intend to do. In fact, in the latter case,
not only we have to accommodate terms $\theta$ and $|\nabla_v q|^2/2$,
which are unusual in any Monte Carlo approach, but also have to resolve the
fact that, in the new model, we do not have a gain-loss structure and
$q$ does not represent a material density. We address these issues
in the next section by focusing on a specific collision model.

\section{A Monte Carlo approach}
\label{sec:MC}

This section is devoted to the formulation of Monte Carlo algorithms
to solve our governing model and the corresponding adjoint problem.
For this purpose, we focus on the collision mechanism proposed
by J. Keilson and J.E. Storer in \cite{KeilsonStorer52} for modelling
Brownian motion. We remark that in subsequent works,
the Keilson-Storer (KS) collision term has been successfully applied
to the estimation of transport
coefficients \cite{Berman1986}, laser spectroscopy \cite{Berman2010}, and molecular dynamics simulations \cite{Tran2009}, reorientation of molecules in liquid water \cite{Gelin2006}, and quantum transport \cite{Kosov2018}. Further, notice that the KS term allows to mimic strong and weak collision limits
\cite{Strekalov2012}. A microscopic derivation of the KS collision
term is presented in \cite{Gelin2015}. Notice that an open-loop optimal
control problem governed by our kinetic model with KS collision
and $F(x,v,t;u)=u(x)$ and
a quadratic $H^1$ control cost (not the expected value) has been
investigated by the authors in \cite{Bartsch2021MOCOKI}.

The KS collision kernel is given by
\begin{align*}
	A(v,w)= \Gamma \sqrt{\frac{\beta}{\pi}} \exp\left( -\beta\, | w-\gamma v| ^2 \right),
\end{align*}
where $\gamma \lessapprox 1$ is a damping parameter. Further, we have
$\beta = \frac{M}{2 k_B T_p \,  (1 - \gamma^2) }$, where $M$ is the mass
of each particle, $T_p$ is their temperature,
$k_B$ is the Boltzmann constant, and $\Gamma$ represents a constant
relaxation rate $1/\tau$ (collision frequency). In \eqref{CollisionTerm},
we have $\sigma=\Gamma$.

The Monte Carlo method is a mesh-less scheme in which the particles are
represented by labelled pointers to structures that contain all
information as velocity, time of collision, etc..
A timestep in a MC procedure consists of
changing the content of this structure, and
adding or subtracting pointers, that is, particles, if required. The
content of the structure is changed subject to two processes.
On the one hand, according to the streaming phenomenon, represented by the operator $L_u$, where position and velocity of each particle
are changed according to the underlying dynamical system.
On the other hand, the velocities are changed whenever a
collision occurs.

In this evolution process, the timestep size $\Delta t$ is chosen some order of magnitude bigger than the mean time between two collisions. This choice
allows to retain the statistical significance of the occurrence of collisions.
However, the timestep size cannot be too large since in this case
transient phenomena would be filtered out.

In order to determine when a particle undergoes a velocity transition due to collision,
one can follow the procedure described in, e.g., \cite{Jacoboni1983}. If $\tau^{-1}$ is the collision frequency, then $\tau^{-1} dt$ is the probability that a particle has a collision during the time $dt$. Now, assuming that a particle has a collision at time $t$, the probability that it will be subject to another collision at time $t + \delta t$
is computed according to a Poisson distribution given by
\begin{equation*}
	\exp\left( -\int_t^{t+\delta t} \! \tau^{-1} \, dt' \right) = \exp(- \delta t/\tau).
\end{equation*}
Thus, following a standard approach, and using a uniformly distributed random number $r$ between 0 and 1, one obtains the following
\begin{equation}
	\delta t=-\tau\log(r).
	\label{eq:delta_t}
\end{equation}
This is the free streaming time within which
the microscopic particles' dynamics is integrated. For this purpose, we apply the symplectic (velocity) Verlet algorithm; see, e.g.,  \cite{Hairer03,Verlet67} and \cite{Skeel97}.

While updating the position, we have to take into account the
boundary of the physical domain with the given (partial) specular reflection boundary conditions. If $\alpha<1$, we have partial absorption
that is implemented by pruning particles from the list with a certain probability corresponding to the parameter $\alpha$.

Assuming that a collision occurs, the new velocities of the
physical particles after the collision are computed based on the collision kernel, which in the KS case can be written as a normal distribution as follows
\begin{equation}
	A(v,w) =\Gamma\mathcal{N}_1\left( \gamma v,\frac{1}{2\beta} \right),
	\label{A_norm}
\end{equation}
where $\mathcal{N}_1$ denotes the univariate normal
Gaussian distribution. Since the covariance matrix is diagonal, it is possible to generate component-wise the new velocity according to the corresponding one-dimensional distributions. For this purpose, the Box-Muller formula \cite{Box1958}.

In the next two sections, we illustrate our Monte Carlo approach to
solve the kinetic model and its corresponding adjoint. We choose a number of particles $N_{f}$, and consider a partition of the time interval $[0,T]$ in $N_t$
subintervals of size $\Delta t = T/N_t$ such that $\Delta t \gg \delta t$.
With this setting, we have $t^k = k \Delta t$, for the time of the $k$-th
timestep, $k=0,\ldots,N_t$. We define $\Gamma_{\Delta t} = \left\lbrace  t^k = k \Delta t \in [0,T], \quad k = 0,\ldots,N_t  \right\rbrace.$

\subsection{Monte Carlo simulation of the kinetic model}

We define $F$ as the list of labelled pointers to structures that
resemble particles.
We denote with $F^k[p]$ the pointer to the $p$-th particle at the
$k$-th timestep. We have $p=1,\ldots,N_f$ and
$k = 0,\ldots,N_t$.
Further, let $F^k[p].v$ be the velocity
of the $p$-th particle at the $k$-th timestep,
and let $F^k[p].x$ be the position of the  $p$-th particle at the $k$-th timestep.
Moreover, let $F^k[p]. t'$ be the time that is elapsed for the $p$-th particle starting from $t^k$.
This quantity is used to determine if the particle will
undergo another collision in the current timestep, assuming
that $0 \leq F^k[p]. t' < \Delta t$.

To initialize $F^0$ using the distribution $f_0$, we apply \cref{algo:initialize} given below.
\begin{algorithm}
	\caption{Initial condition}
	\label{algo:initialize}
	\begin{algorithmic}[1]
		\REQUIRE $f_0(x,v)$
		\vspace{0.1cm}
		\FOR{$p=1$ \TO $N_{f}$}
		\vspace{0.1cm}
		\STATE Compute $(F^0[p].v,F^0[p].x) \sim f_0$
		\vspace{0.1cm}
		\STATE Set $F^0[p].  t' = 0$
		\vspace{0.1cm}
		\ENDFOR
	\end{algorithmic}
\end{algorithm}

The (partial) specular reflecting boundary conditions (for the case of a one-dimensional boundary $\Omega \subset \RR$) are implemented as given in \cref{algo:bdry}.

\begin{algorithm}
	\caption{Specular reflecting boundary condition}
	\label{algo:bdry}
	\begin{algorithmic}[1]
		\REQUIRE Updated position $\tilde{x}$ according  to Verlet method
		\vspace{0.1cm}
		\IF { $\tilde{x} \in \Omega = [0,L]$}
		\vspace{0.1cm}
		\RETURN $\tilde{x}$
		\vspace{0.1cm}
		\ELSE
		\vspace{0.1cm}
		\STATE Sample uniform random number $\xi \sim \mathcal{U}(0,1)$
		\vspace{0.1cm}
		\IF {$\xi > \alpha$ }
		\vspace{0.1cm}
		\STATE Prune particle from list of particles
		\vspace{0.1cm}
		\ELSE
		\vspace{0.1cm}
		\IF { $\tilde{x} < 0$}
		\vspace{0.1cm}
		\STATE Calculate $\omega = \lfloor \tilde{x}/L \rfloor\mod 2$
		\vspace{0.1cm}
		and set velocity $v = (-1)^{\omega-1}v$
		\vspace{0.1cm}
		\RETURN $x = \omega L + (-1)^\omega (-\tilde{x} \mod L)$
		\vspace{0.1cm}
		\ELSIF { $\tilde{x} > L$}
		\vspace{0.1cm}
		\STATE Calculate $\omega = \lfloor \tilde{x}/L \rfloor\mod 2$
		\vspace{0.1cm}
		and set velocity $v = (-1)^{\omega-1}v$
		\vspace{0.1cm}
		\RETURN $x = (1-\omega)L + (\omega-1)(\tilde{x} \mod L)$
		\vspace{0.1cm}				
		\ENDIF
		\vspace{0.1cm}
		\ENDIF
		\vspace{0.1cm}
		\ENDIF
	\end{algorithmic}
\end{algorithm}

As already mentioned, the MC simulation of our kinetic model \eqref{eq:linear_transport_general}
consists of two processes. On the one hand, we simulate
the streaming phenomenon where position and velocity
are changed according to the underlying dynamical system.
On the other hand, the velocity
of each particle may change due to collisions. In this case, a particle with position $x$ and velocity $v$ that is subject to collision acquires a new velocity $v'$
by keeping approximately the same position. Notice that
choosing a numerical boundary for the velocity large enough,
the probability that the velocity of a particle exceeds this boundary is very low
but possibly not zero. If this rare event happens, one generates
again a new velocity for the particle using the same pre-collision velocity as before.

Our Monte Carlo simulation algorithm can be summarized
with \cref{algo:DSMC_Boltz}.

\begin{algorithm}
	\caption{MC solver of the kinetic model}
	\label{algo:DSMC_Boltz}
	\begin{algorithmic}[1]
		\REQUIRE  $f_0(x,v)$, $u(x,v,t)$
		\vspace{0.1cm}
		\STATE Initialise $N_{f}$ particles using \cref{algo:initialize} and $f_0(x,v)$, set $\delta t_2 = 0$
		\vspace{0.1cm}
		\FOR{$k=0$ \TO $N_t-1$}
		\vspace{0.1cm}
		\FOR{$p=1$ \TO $N_{f}$}
		\vspace{0.1cm}
		\WHILE{$F^k[p].t'<\Delta t$}
		\vspace{0.1cm}
		\STATE Compute $\delta t_1$ according to \eqref{eq:delta_t}
		\vspace{0.1cm}
		\STATE Determine $F^k[p].v \sim \mathcal{N}\left(\gamma v,\frac{1}{2\beta}\right)$
		\vspace{0.1cm}
		\STATE update $F^k[p].x$ and $F^k[p].v$ according to the Verlet-Algorithm: \\
		\vspace{0.1cm}
		$F^k[p].x \, = \, F^k[p].x + F^k[p].v\,\delta t_1 \, + \, u(F^k[p].x) \frac{\delta t_1+\delta t_2}{2}\delta t_1$, \\
		\vspace{0.1cm}
		$F^k[p].v \, = \,  F^k[p].v + u( F^k[p].x)\,\delta t_1$ \\
		\vspace{0.1cm}
		and taking the boundary condition into account using \cref{algo:bdry}
		\vspace{0.1cm}
		\STATE $F^k[p].t'=F^k[p].t'+\textcolor{black}{\delta t_1}$
		\vspace{0.1cm}
		\STATE $\delta t_2 = \delta t_1$
		\vspace{0.1cm}
		\ENDWHILE
		\vspace{0.1cm}
		\IF{$F^k[p].t' > \Delta t$}
		\vspace{0.1cm}
		\STATE $F^{k+1}[p].t' = F^k[p].t' \text{ mod }\Delta t$
		\vspace{0.1cm}
		\ENDIF
		\vspace{0.1cm}
		\ENDFOR
		\vspace{0.1cm}
		\ENDFOR
	\end{algorithmic}
\end{algorithm}

Notice that in our one-shot method there is no need to solve
the kinetic model. However, the solution of this model is required to
simulate the evolution of the density with the computed control and
for evaluating the cost functional.

\subsection{MC simulation of the adjoint kinetic model}

We illustrate a reformulation of our augmented adjoint model \eqref{eq:AdjointAugmented}
that makes it amenable to be solved with a Monte Carlo method.
This aim requires to interpret the adjoint problem as a kinetic model
for `adjoint' particles, and the adjoint variable $q$ as their density.

The procedure that follows is tailored to the case of a KS collision term.
We define
\begin{align*}
	C^*_0 :=  \SPI{d} \big( A(w,v)-A(v,w)\big)\,dw = \Gamma \, \frac{1- \gamma}{\gamma}.
\end{align*}
Further, we introduce the following `adjoint' collision term
\begin{align*}
	C^*[q](x,v,t) = \SPI{d} A^*(w,v) \, q(x,w,t)\,dw - q(x,v,t) \, \SPI{d} A^*(v,w)\,dw,
\end{align*}
where $A^*(w,v)  = \frac{1}{\gamma} \, A(v,w)$, which results in an `adjoint'
collision frequency $1/\tau_q$ where $\tau_q = \gamma \, \tau$. Correspondingly, formula \eqref{eq:delta_t} applies to compute
the free streaming time of the adjoint particles.

With this setting, we can write our adjoint equation as follows:
\begin{align}
	-\partial_t q(x,v,t)+ L^*_uq(x,v,t) = C^*[q](x,v,t) + C^*_0 \, q(x,v,t)
	-\theta(x,v,t) {- \frac{1}{2\nu}| \nabla_v q(x,v,t)|^2}.
	\label{eq:Adjoint_feedback_long}
\end{align}

Now, we focus on the collision term in this equation and notice that
the transition probability is given by
\begin{align*}
	A^*(v,w)
	=\Gamma \mathcal{N}_1\left(\frac{v}{\gamma},\frac{1}{2\beta\gamma^2}  \right).
\end{align*}
With this knowledge, we can implement collisions using the Box-Muller
formula as discussed above.

Notice that, in the case of the adjoint kinetic model, we have to
deal also with particles having large velocity. In fact,
because of the structure of the adjoint collision term and the high variances of
$\theta$ and $\varphi$, the occurrence of passing a numerical bound
by the post-collision velocity of an adjoint particle is not a rare event.
In this case, the adjoint particle is removed from the list.

In addition to the simulation procedure for the gain-loss structure in the
kinetic model, we need to implement
the linear reaction term $C_0^*q$ appearing in the adjoint equation.
For this purpose, a splitting strategy leads
to consider the following differential equation
\begin{align*}
	-\partial_t q(x,v,t) = C_0^*\,q(x,v,t) .
\end{align*}
In terms of a small timestep ${\Delta t} >0$ depending on the collision frequency, this equation can be approximated by
\textcolor{black}{Euler's method as}
\begin{align*}
	q(x,v,t-{\Delta t}) = q(x,v,t) +{\Delta t}\, C_0^*\, q(x,v,t) .
\end{align*}
Therefore we can implement the
contribution of the linear reaction term as an increase in the
backward evolution of the distribution of
the adjoint particles at every point according to the
factor $(1+{\Delta t}\,C_0^*)$ for each timestep.
This means that one particle with velocity $v$
at time $t$ is replaced by $(1+{\Delta t}\,C_0^*)$ particles with the same velocity at the time $t - {\Delta t}$.

The MC procedure to implement the reaction term is presented in \cref{algo:linReact_creation}.
Analogously to $F$, we denote with $Q$ the list of labelled pointers
to structures representing adjoint particles.
\begin{algorithm}
	\caption{Implementation of the linear reaction term}
	\label{algo:linReact_creation}
	\begin{algorithmic}[1]
		\REQUIRE $Q^k$, $N^k_{q}$
		\vspace{0.1cm}
		\STATE Set $N \in \NN_0,$  $\varepsilon \in [0,1)$ such that
		$\textcolor{black}{\Delta t} \, C_0^* = N + \varepsilon$
		\vspace{0.1cm}
		\FOR{$p=1$ \TO $N^k_{q}$}
		\STATE Generate $N$ particles with the velocity and position
		$(Q^k[p].x,Q^k[p].v)$
		\vspace{0.1cm}
		\STATE Generate uniform random number $r \in [0,1]$
		\vspace{0.1cm}
		\IF{$r > 1-\varepsilon$}
		\vspace{0.1cm}
		\STATE Generate a particle with the velocity and position  $(Q^k[p].x,Q^k[p].v)$
		\vspace{0.1cm}
		\ENDIF
		\vspace{0.1cm}
		\ENDFOR
		\vspace{0.1cm}
		\STATE Add generated particles to the existing ones in $Q^k$
	\end{algorithmic}
\end{algorithm}

Next, we discuss our implementation of the source term
$-\left(\theta + \frac{1}{2\nu} | \nabla_v q|^2 \right)$.
In our MC algorithm, the source term corresponds to adding or
subtracting a number of particles according to the value given by
$-\left(\theta + \frac{1}{2\nu}|\nabla_v q|^2\right)$. However, for the evaluation
of this term, we need a computational mesh in phase space in order
to compute $\theta$ and $\nabla_v q$.

For this purpose, we consider a bounded domain of velocities
$\Upsilon \coloneqq [-v_{\max},v_{\max}] \subset \RR$,
where $v_{\max}>0$ is a working parameter that represents a
maximum value for each component of the velocities of the particles.
This parameter is assumed to be taken large enough, such that the exceeding of $[-v_{\max},v_{\max}]$ by the velocity of particles is a rare event.

Now, we define a partition of $\Upsilon$ in equally-spaced,
non-overlapping square cells with side
$\Delta v = {2\textcolor{black}{v_{\max}}}/{N_v}$ where $N_v \geq 2$. On this partition,
we consider a cell-centred representation of the velocities as follows:
\begin{align*}
	\Upsilon_{\Delta v}= \left\lbrace  v_j \in \Upsilon, \quad j = 1,\ldots,N_v \right\rbrace, \qquad v_j =  \left(j-\frac{1}{2}\right)\Delta v - v_{\max}
\end{align*}
and analogously we define a grid on the spatial domain
\begin{align*}
	\Omega_{\Delta x} = \left\lbrace  x_i \in \Omega, \quad i = 1,\ldots,N_x  \right\rbrace, \qquad x_i = \left(i-\frac{1}{2}\right)\Delta x .
\end{align*}
Hence, we have the discretized phase space $\Omega_{\Delta x} \times \Upsilon_{\Delta v}$.

On the given mesh, the computation of $\theta$ is straightforward. On
the other hand, the calculation of $\nabla_v q$ requires additional work.
In fact, the function $q$ on the grid needs to be assembled using
the list of particles $Q^k$.
For this reason, we construct the occupation number
$q_{ij}^k$, corresponding to
a cell centred in $(x_i,v_j)$ in the phase space, by
counting the adjoint particles in this cell at timestep $k$. Thus, we have
\begin{align}
	q_{ij}^k = \sum_{p=1}^{N_q^k}  \textcolor{black}{\mathbbm{1}}_{ij}\left( Q^k[p].x, Q^k[p].v \right) ,
	\label{eq:assembling_Distribution_q}
\end{align}
where ${\mathbbm{1}}_{ij}(\cdot,\cdot)$ denotes the indicator function, i.e. ${\mathbbm{1}}_{ij}(x,v) = 1$ if and only if $(x,v)$ is located
in the cell of $\Omega_{\Delta x}\times\Upsilon_{\Delta v}$
centred at $(x_i,v_j)$, and zero otherwise.

Clearly, a similar procedure has to be performed in order to reconstruct
the density $f$ on the grid as follows:
\begin{align}
	f_{ij}^k \coloneqq \sum_{p=1}^{N_f}  \textcolor{black}{\mathbbm{1}}_{ij}
	\left(
	F^k[p].x,F^k[p].v \right) .
	\label{eq:assemblin
		g_Distribution_f}
\end{align}

The evaluation of the derivative $\nabla_v q$ is based
on the reconstructed $q_{ij}^k$ given
by \eqref{eq:assembling_Distribution_q}. However,
before finite-difference differentiation is applied, we perform a
denoising step on $q_{ij}^k$.

Our denoising procedure can be put in the framework of Tikhonov
regularization techniques \cite{LakshmiParvathy12},
in the sense that the resulting smooth $\tilde q$ is defined
as the solution to the following minimization problem
\begin{align}
	\label{eq:DenoisingProblem}
	\min \Big(
	\frac{c_s}{2}\int_\Upsilon|\nabla_v \tilde q|^2\,dv \,+\,
	\frac{1}{2} \int_\Upsilon\left(\tilde q -q \right)^2\,dv  \Big) ,
\end{align}
where $q$ is the original noisy adjoint variable given by \eqref{eq:assembling_Distribution_q},
and $c_s>0$ is a regularization parameter that is usually small compared to the maximum value of $q$.

The Euler-Lagrange equation corresponding
to the optimization problem \eqref{eq:DenoisingProblem} is given by
\begin{align}
	\label{eq:ELequationDenoising}
	-c_s \, \Delta_v \tilde{q} + (\tilde{q}- q) = 0,
\end{align}
where $\Delta_v = \partial^2_{vv}$ is approximated by finite differences,
and we impose homogeneous Neumann boundary conditions, which
guarantees that the total (adjoint) mass is conserved:
\begin{align*}
	\int_\Upsilon \tilde{q} \, dv \,=\,
	\int_\Upsilon q \,dv.
\end{align*}
This fact is proved by applying the divergence theorem.

Now, we can evaluate $\nabla_v \tilde q$ as follows:
\begin{equation}
	\nabla_v \tilde q =   \frac{\tilde{q}^k_{i,j+1}-\tilde{q}_{i,j-1}^k}{2 \Delta v}
	\qquad i \in [N_x],\, j \in [N_v-1], \, k \in [N_t],
	\label{eGU}
\end{equation}
where $[N]$ denotes the set $\{1,\ldots,N\}$. For $j=N_v$, we use the backward finite-difference approximation of $\nabla_v q$.

With this preparation, we can present the MC implementation of the source term
in \cref{algo:DSMC_creation}.
\begin{algorithm}
	\caption{Implementation of the source term}
	\label{algo:DSMC_creation}
	\begin{algorithmic}[1]
		\REQUIRE $Q^k$, $\theta$
		\STATE {Assemble the distribution function $q^k$ from $Q^k$ using
			\eqref{eq:assembling_Distribution_q}; \\
			\vspace{0.1cm}
			compute the corresponding
			smoothed adjoint variable $\tilde{q}^k$ by solving \eqref{eq:ELequationDenoising}.}
		\vspace{0.1cm}	
		\FOR{{$i = 1$ \TO $N_x-1$}}
		\vspace{0.1cm}
		\FOR{{$j = 1$ \TO $N_v-1$}}
		\vspace{0.1cm}
		\STATE Calculate the value $N_{\theta} = -\theta(x_i,v_j,t^k)$ in the cell $(x_i,v_j)$ according to \eqref{eq:Structure_theta}.
		\vspace{0.1cm}
		\STATE {Set {$N_q = \frac{1}{2\nu} |(\tilde{q}^k_{i,j+1}-\tilde{q}^k_{i,j-1})/(2 \Delta v)|^2 $}}.
		\vspace{0.1cm}
		\STATE Generate $N_{new} = \max(\left\lfloor N_{\theta} - {N_q} \right\rfloor,0)$ new particles with velocity and position components having the uniform distribution within the phase space cell centred in $(x_i,v_j)$: $ (x,v) \sim \mathcal{U}_{\Delta x, \Delta v}\left( x_i,v_j \right) $
		\ENDFOR
		\vspace{0.1cm}
		\ENDFOR
		\STATE Add these particles to the existing ones in $Q^k$
	\end{algorithmic}
\end{algorithm}

Notice that by the definition of $N_{new}$ particles are only created but never removed from the list. This is of course an approximation, but it is a sufficient good one in order to calculate a effective control, as it is evident by the experiments.

Since the implementation of reaction and source terms leads
to varying the number of adjoint particles in time,
we index this number with $k$ and write $N_q^k$.

In order to complete the description of our algorithm for
the augmented adjoint problem, we remark that the terminal condition
$q=-\varphi(x,v)$ must be implemented. This procedure corresponds
to initializing $Q^{N_t}$ using the distribution $-\varphi$ in phase space (cf. \cref{algo:initialize}).

The MC simulation algorithm for the adjoint problem is given in \cref{algo:DSMC_adj}.

\begin{algorithm}
	\caption{MC solver of the augmented adjoint kinetic model}
	\label{algo:DSMC_adj}
	\begin{algorithmic}[1]
		\REQUIRE $\theta(x,v,t)$, $\varphi(x,v)$, $u(x,v,t)$.
		\vspace{0.1cm}
		\STATE Initialize $Q^{N_t}$ with $N_{q}^{N_t}$ particles using \cref{algo:initialize} and $-\varphi$, set $\delta t_2 = 0$
		\vspace{0.1cm}
		\FOR {$k=N_t$ \TO $1$}
		\vspace{0.1cm}
		\STATE {Use $Q^k$ and \cref{algo:DSMC_creation} to implement the source term}
		\vspace{0.1cm}
		\STATE Use $Q^k$ and \cref{algo:linReact_creation} to implement the linear reaction term
		\vspace{0.1cm}
		\FOR{$p=1$ \TO $N_{q}^k$}
		\vspace{0.1cm}
		\WHILE{$Q^k[p].t'<\Delta t$}
		\vspace{0.1cm}
		\STATE Generate $\delta t_1 $ according to \eqref{eq:delta_t} using $\tau_q$ instead of $\tau$
		\vspace{0.1cm}
		\STATE Determine $v\sim \mathcal{N}\left(\frac{v}{\gamma},\frac{1}{2\beta\gamma^2}\right)$
		\vspace{0.1cm}
		\STATE update $Q^k[p].x$ and $Q^k[p].v$ according the adjoint Verlet-Algorithm:
		\vspace{0.1cm}
		\STATE $Q^k[p].x \, = \, Q^k[p].x + Q^k[p].v\,\delta t_1 \, - \, u(Q^k[p].x) \frac{\delta t_1+\delta t_2}{2}\delta t_1$,
		\vspace{0.1cm}
		\STATE $ Q^k[p].v \, = \,  Q^k[p].v - u( Q^k[p].x)\,\delta t_1$,
		and taking the adjoint boundary condition into account using \cref{algo:bdry}
		\vspace{0.1cm}
		\STATE $Q^k[p].t'=Q^k[p].t'+\delta t$
		\vspace{0.1cm}
		\STATE $\delta t_2 = \delta t_1$
		\vspace{0.1cm}
		\ENDWHILE
		\vspace{0.1cm}
		\IF{$Q^k[p].t' > \Delta t$}
		\vspace{0.1cm}
		\STATE $Q^{k-1}[p].t' = Q^k[p].t' \text{ mod }\Delta t$
		\vspace{0.1cm}
		\ENDIF
		\vspace{0.1cm}
		\ENDFOR
		\vspace{0.1cm}
		\ENDFOR
	\end{algorithmic}
\end{algorithm}

\begin{remark}
	Notice that we do not need to formulate any iterative gradient-based method since we only need to solve \eqref{eq:AdjointAugmented} once in order to establish the control given by \eqref{eq:ReducedGradient}.
	The numerical procedure to solve \eqref{eq:AdjointAugmented} is summarized in \cref{algo:DSMC_adj}.
\end{remark}

\section{Numerical experiments}
\label{sec:NumericalExamples}

In this section, we present results of numerical experiments
demonstrating the effectiveness of our framework to design a
control field capable of driving an initial density of
particles randomly distributed in the phase space to reach and
maintain a desired cyclic trajectory.

We consider a two dimensional phase space domain $\Omega = [0,p_{\max}] \times [-v_{\max},v_{\max}]$ with positive $p_{\max}$, $v_{\max}$.
On this domain, we set $f_0$ equal to a uniform distribution.

The desired orbit corresponds to a harmonic oscillator of unit mass
and force corresponding to Hooke's law as follows
\begin{align*}
	F_0(x,v) = - \omega^2\left(x - \frac{p_{\max}}{2}\right) .
\end{align*}
The resulting trajectory is given by
\begin{align*}
	z_D(t) =  \binom{x_D(t)}{v_D(t)} =
	\binom{ 2.5\cos(\omega t )+x_0}
	{-2.5\omega\sin(\omega t)-v_0},
	\qquad
	\omega = \frac{2\pi}{T},
	\qquad
	\binom{x_0}{v_0} = \binom{p_{\max}/2}{0},
\end{align*}
where $T$ is the period of the orbit.

In order to determine the control field everywhere in the computational
phase space domain, we consider an initial density that is uniformly
distributed. Therefore we take $\bar{\theta}$ given by \eqref{eThetaStat}.
We initialize the iterative solver with $u^0\equiv0$, and the values of the
physical and numerical parameters are specified in \cref{tab:ParamsEx2}.
We present results of different experiments, where we choose $c_s = 0.5$.
This value is based on our numerical experience showing that the accuracy
of the control function computed by our MC scheme is sensitive to the value of $c_s$
when using a relatively small number of adjoint particles $N_q$ (as we do).

\begin{table}
		\centering
		\begin{tabular}{|c|c||c|c|}
			\hline
			\textbf{Symbol} & \textbf{Value} & \textbf{Symbol} & \textbf{Value} \\[0.1cm]
			\hline
			$N_t$ & {100} & $\Delta t$  & $0.025$  \\[0.1cm]
			\hline
			$N_x \times N_v$ [-]  & $50 \times 50$ & $v_{\max}$ & $5$ \\[0.1cm]
			\hline
			$p_{\max}$  &  $10.0$   & $\Delta v$ & $0.2$   \\[0.1cm]
			\hline
			$\Delta p$  &  $0.2$     & $N_f$  &  $10^4$  \\[0.1cm]
			\hline
		\end{tabular}
			\bigskip
		\begin{tabular}{|c|c||c|c|}
			\hline
			\textbf{Symbol} & \textbf{Value} & \textbf{Symbol} & \textbf{Value} \\
		\hline
		$\gamma$   &  $0.9999$ & $\alpha$  & $0.5$   \\[0.1cm]
		\hline
		$\nu$   & $1$  &$C_\theta$, $C_\varphi$   & $10^3$   \\[0.1cm]
		\hline
		$c_s$ & $0.5$  & $N_q^{N_t}$  & $6 \cdot 10^2$   \\[0.1cm]
		\hline
	\end{tabular}
\caption{ Numerical and physical parameters. }
\label{tab:ParamsEx2}
\end{table}

With this setting, we obtain the control field $u=u(x,v,t)$ that is depicted in
\cref{fig:Control_quiver} over several timesteps from $t=0$ to $t=T$,
where $T=2.5$.
In \cref{fig:control_uniform_case}, one can see the evolution of the density $f$ subject to the action of the computed control; the plots in the
these two figures refer to the same timesteps.

\begin{figure}[h]
\begin{subfigure}[l]{\scaleQuiverPlots\textwidth}
	\center
	\includegraphics[width=\textwidth]{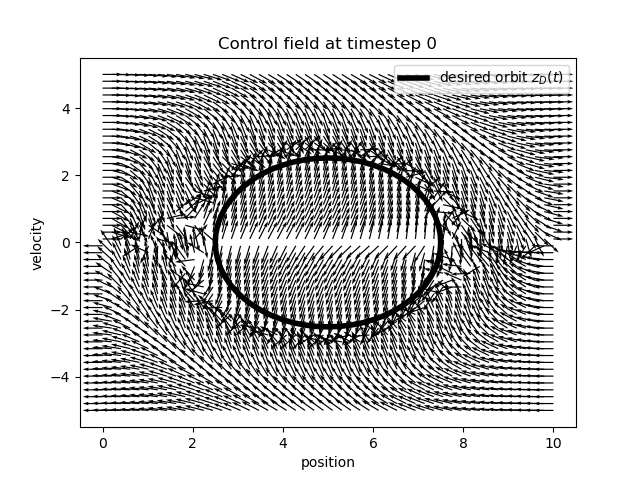}
\end{subfigure}
\begin{subfigure}[l]{\scaleQuiverPlots\textwidth}
	\center
	\includegraphics[width=\textwidth]{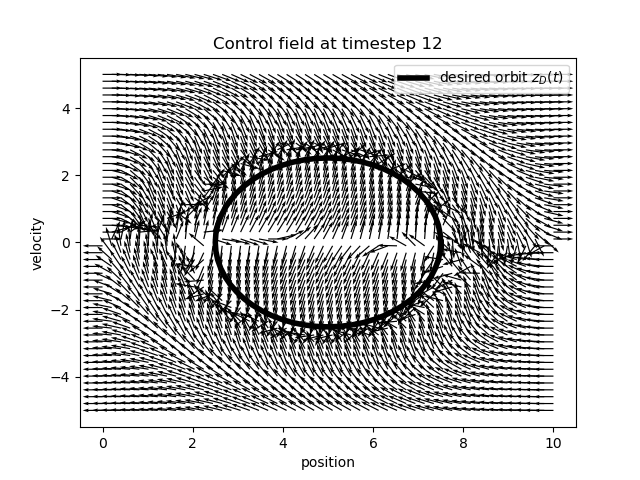}
\end{subfigure}
\begin{subfigure}[r]{\scaleQuiverPlots\textwidth}
	\center
	\includegraphics[width=\textwidth]{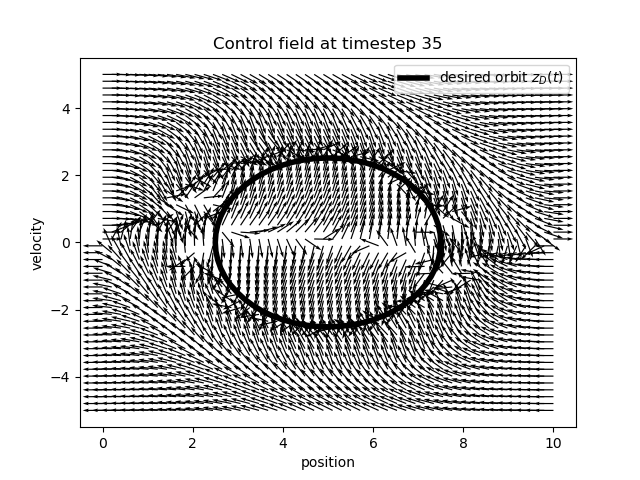}
\end{subfigure}
\begin{subfigure}[l]{\scaleQuiverPlots\textwidth}
	\center
	\includegraphics[width=\textwidth]{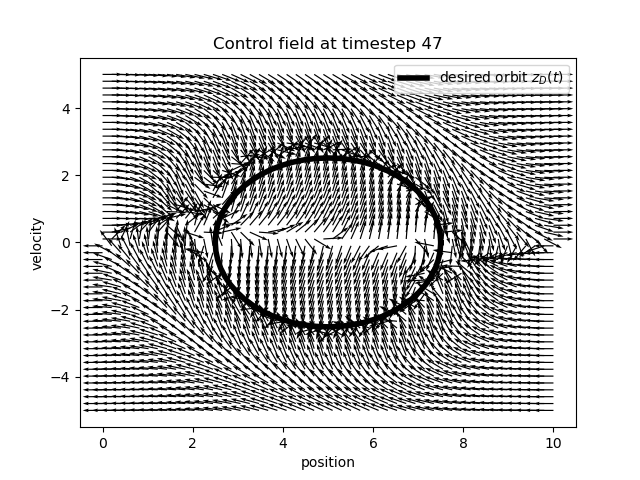}
\end{subfigure}
\begin{subfigure}[l]{\scaleQuiverPlots\textwidth}
	\center
	\includegraphics[width=\textwidth]{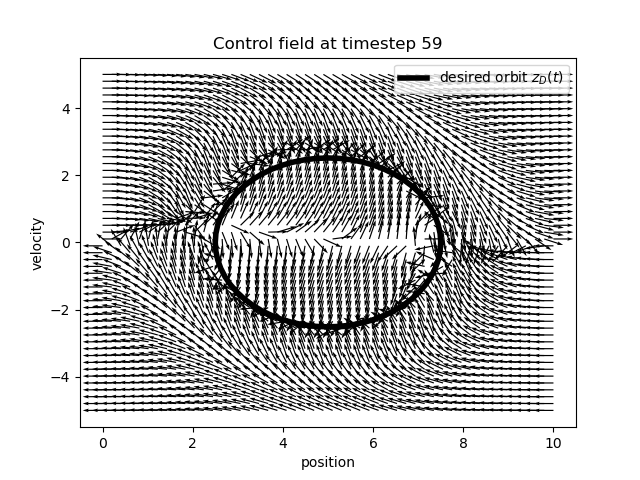}
\end{subfigure}
\begin{subfigure}[r]{\scaleQuiverPlots\textwidth}
	\center
	\includegraphics[width=\textwidth]{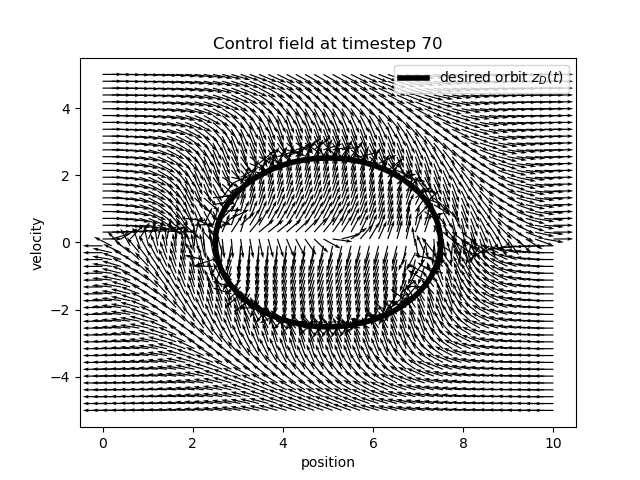}
\end{subfigure}
\begin{subfigure}[l]{\scaleQuiverPlots\textwidth}
	\center
	\includegraphics[width=\textwidth]{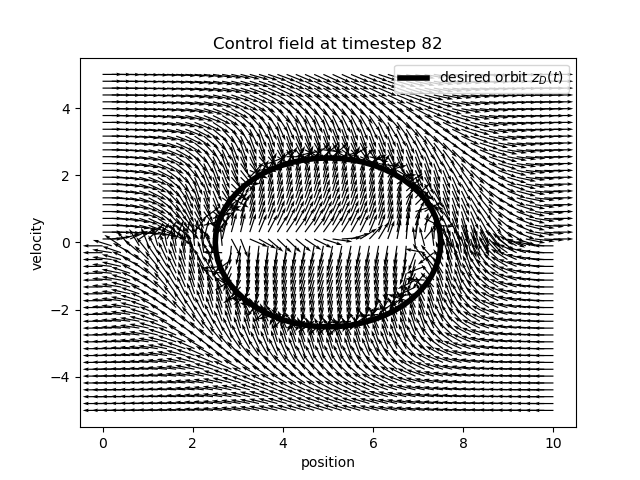}
\end{subfigure}
\hfill
\begin{subfigure}[l]{\scaleQuiverPlots\textwidth}
	\center
	\includegraphics[width=\textwidth]{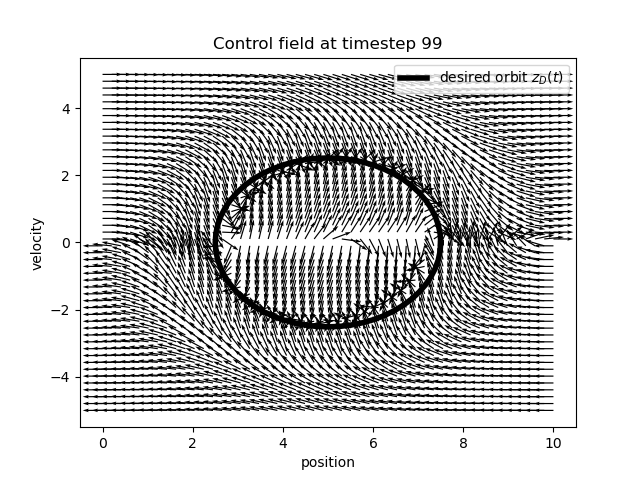}
\end{subfigure}
\vspace{0.2cm}
\caption{Quiver plot of the calculated control. The solid ellipse is the curve $z_D(t)$, $t\in[0,T]$.
	The arrows are given by the scaled vector $(v,u(x,v,t))^T$.}
\label{fig:Control_quiver}
\end{figure}

\clearpage

\begin{figure}[h]
	\includegraphics[width=\scaleParticlePlots\textwidth]{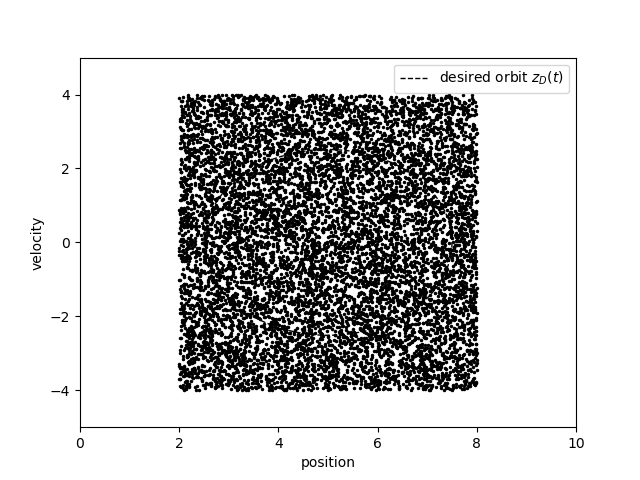}
	\includegraphics[width=\scaleParticlePlots\textwidth]{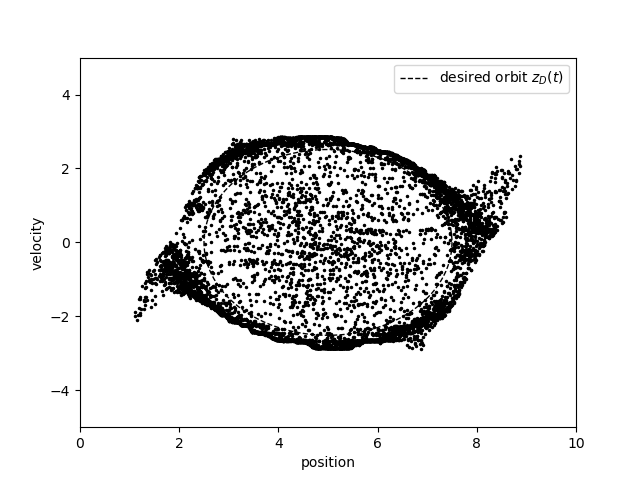}
	\includegraphics[width=\scaleParticlePlots\textwidth]{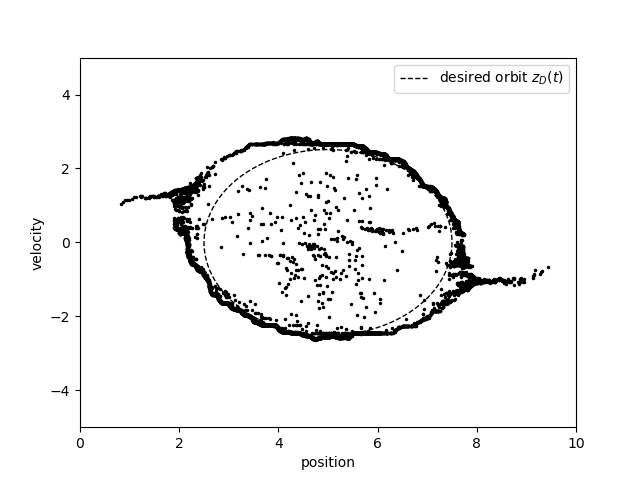}
	\includegraphics[width=\scaleParticlePlots\textwidth]{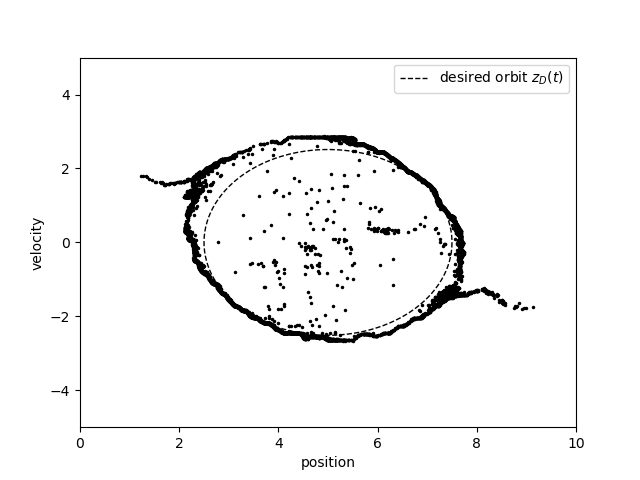}
	\includegraphics[width=\scaleParticlePlots\textwidth]{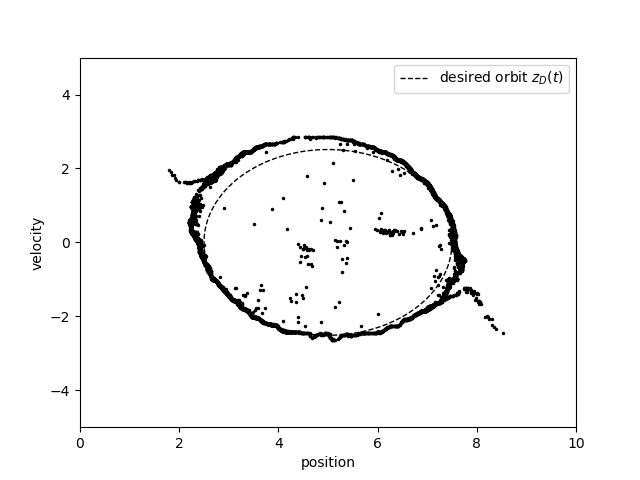}
	\includegraphics[width=\scaleParticlePlots\textwidth]{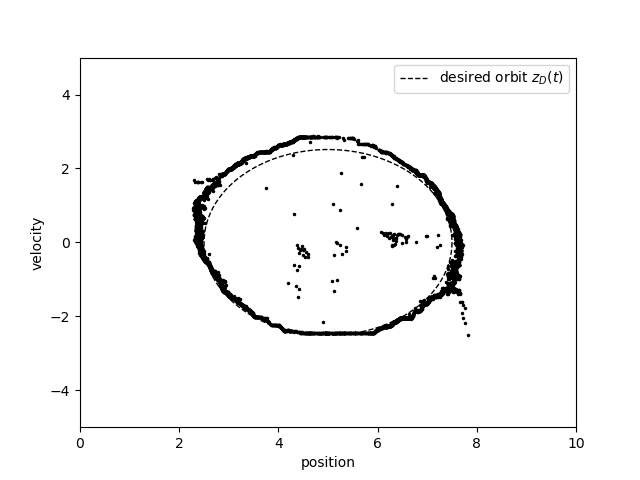}
	\includegraphics[width=\scaleParticlePlots\textwidth]{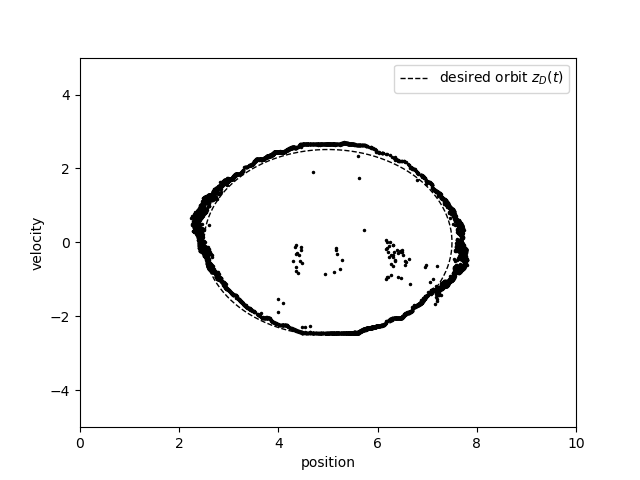}
	\hfill
	\includegraphics[width=\scaleParticlePlots\textwidth]{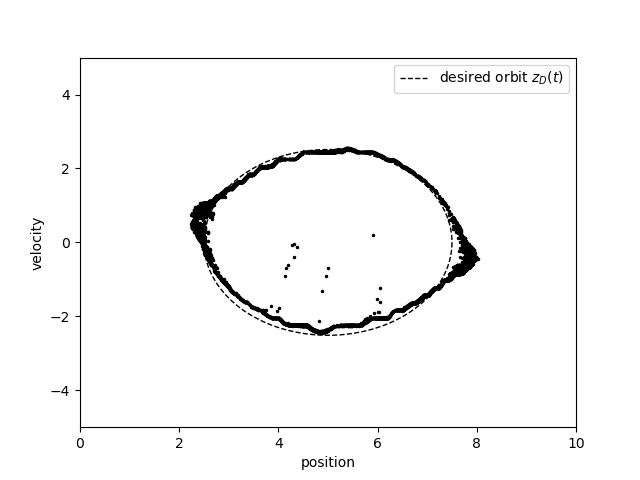}
\caption{Evolution of $f$ starting from a uniform initial distribution and subject to the control field $u$.}
\label{fig:control_uniform_case}
\end{figure}

\clearpage

Next, we present results of experiments to investigate the sensitivity of our computational procedure with respect to the denoising parameter $c_s$ and the weight of the control $\nu$.
In \cref{fig:DifferentDenoising}, we plot the resulting control and the corresponding particle distribution at final time for different values of the denoising parameter $c_s$.
We see that the resulting control field is sensitive to the value of this parameter. However, this sensitivity weakens by choosing a larger number of
particles in the MC implementation of the adjoint model. In fact, a larger number of particles results in a more regular distribution and hence less need of denoising.
In the current case, the choice $c_s=0.5$ appears optimal.

In \cref{fig:DifferentControlWeights}, we plot the control at the final time for different values of the control weight $\nu$ and fixed $c_s=0.5$.
These results show that the resulting control field is less sensitive to the choice of the value of $\nu$.
Notice that for higher values of the control weight we have a weaker control field at $v=0$.

\begin{figure}[h]
	\centering
\begin{subfigure}[l]{\textwidth}
	\includegraphics[width=\scaleQuiverPlotsSmall\textwidth]{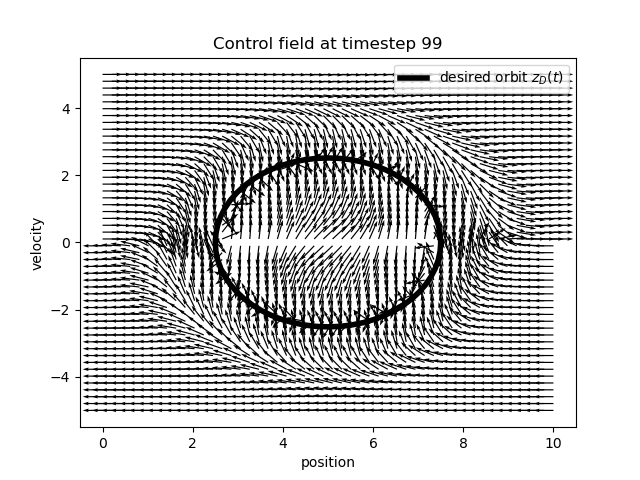}
	\hfill
	\includegraphics[width=\scaleQuiverPlotsSmall\textwidth]{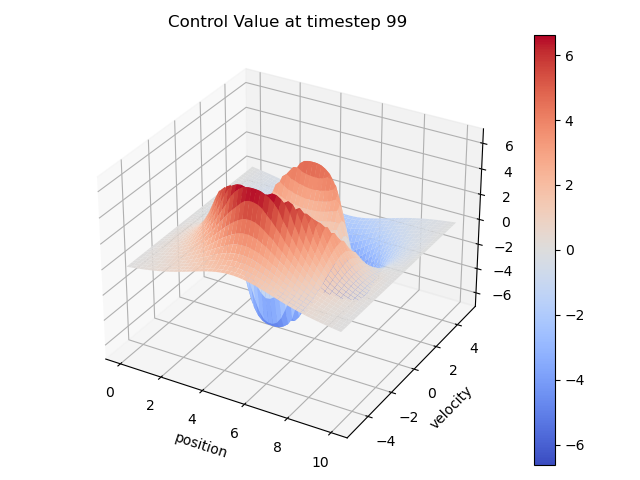}
	\hfill
	\includegraphics[width=\scaleQuiverPlotsSmall\textwidth]{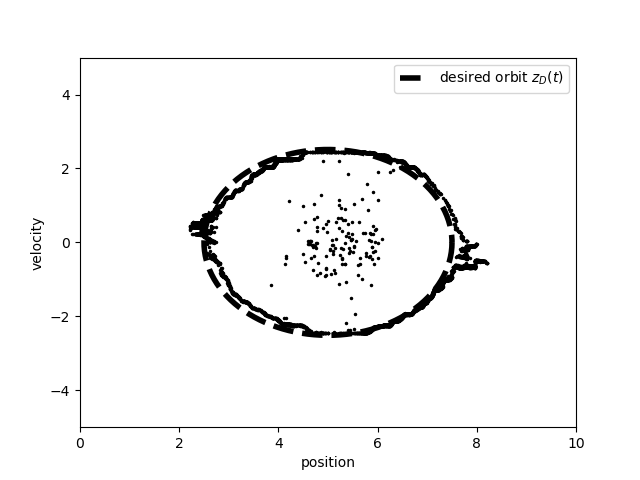}
	\caption{$c_s = 0.1$}
\end{subfigure}
\begin{subfigure}[l]{\textwidth}
	\includegraphics[width=\scaleQuiverPlotsSmall\textwidth]{fig/uniform_initial/direct_calculation/quivers_Control_100/scaled_quiver_Control_it0_99.png}
	\hfill
	\includegraphics[width=\scaleQuiverPlotsSmall\textwidth]{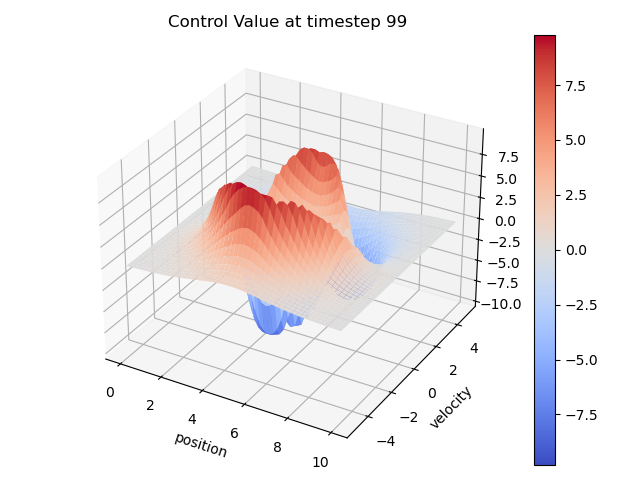}
	\hfill
	\includegraphics[width=\scaleQuiverPlotsSmall\textwidth]{fig/uniform_initial/direct_calculation/animation_state/particleTrajectories_99.png}
	\caption{$c_s = 0.5$}
\end{subfigure}
\begin{subfigure}[l]{\textwidth}
	\includegraphics[width=\scaleQuiverPlotsSmall\textwidth]{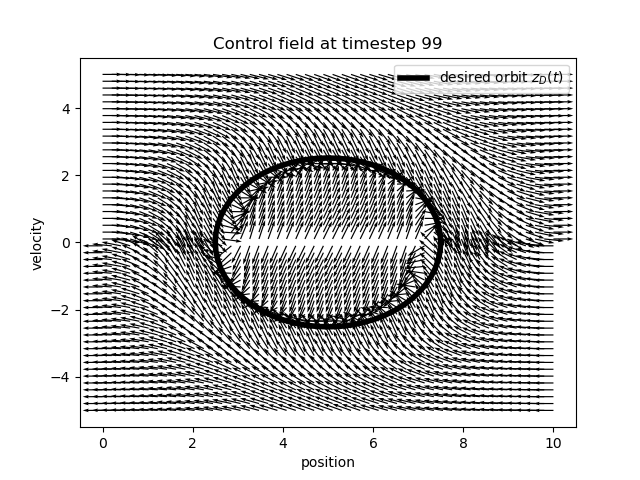}
	\hfill
	\includegraphics[width=\scaleQuiverPlotsSmall\textwidth]{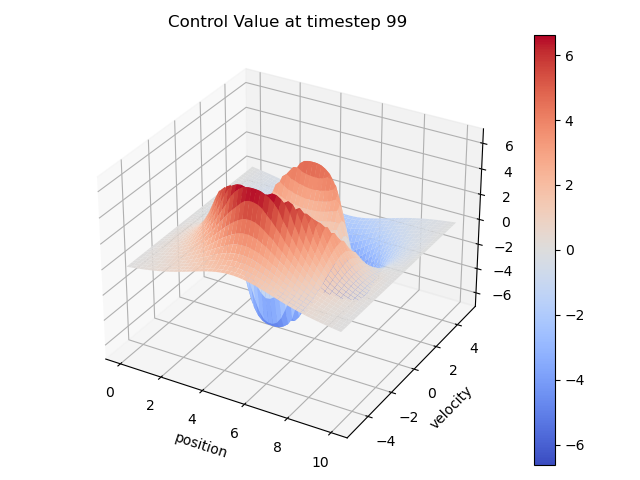}
	\hfill
	\includegraphics[width=\scaleQuiverPlotsSmall\textwidth]{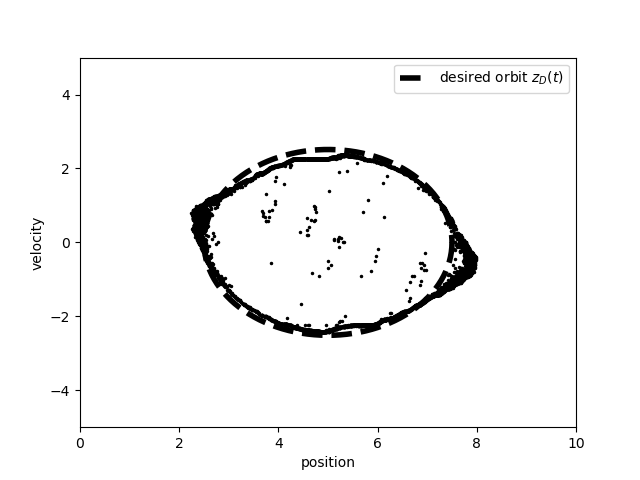}
	\caption{$c_s = 1$}
\end{subfigure}
\caption{Control field and corresponding distribution of particles at final time using different values of the denoising parameter $c_s$.}
\label{fig:DifferentDenoising}
\end{figure}

\clearpage

\begin{figure}[h]
\begin{center}
	\begin{subfigure}[l]{0.24\textwidth}
		\includegraphics[width=\textwidth]{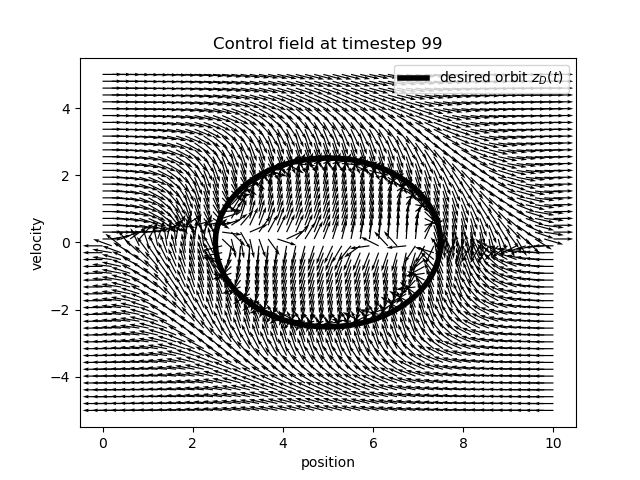}
		\caption{$\nu=10^5$}
	\end{subfigure}
	\begin{subfigure}[l]{.24\textwidth}
		\includegraphics[width=\textwidth]{fig/uniform_initial/direct_calculation/quivers_Control_100/scaled_quiver_Control_it0_99.png}
		\caption{$\nu=1$}
	\end{subfigure}
	\begin{subfigure}[l]{.24\textwidth}
		\includegraphics[width=\textwidth]{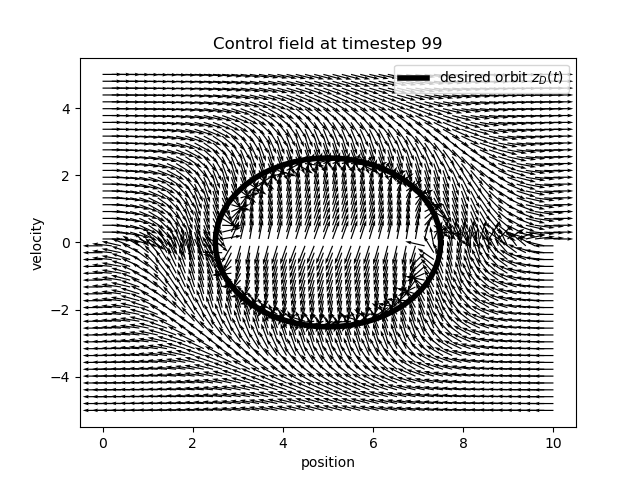}
		\caption{$\nu=10^{-2}$}
	\end{subfigure}
	\begin{subfigure}[l]{.24\textwidth}
		\includegraphics[width=\textwidth]{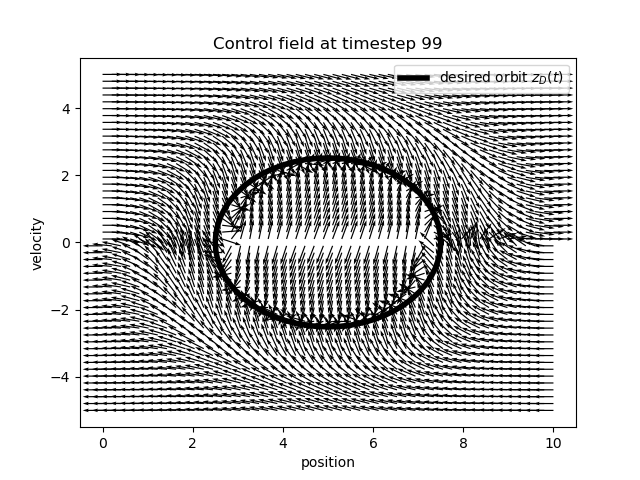}
		\caption{$\nu=10^{-5}$}
	\end{subfigure}
\end{center}
\caption{Control field at final time for different values of the control weight $\nu$. }
\label{fig:DifferentControlWeights}
\end{figure}


In the next experiment, we demonstrate that starting from a
different initial distribution and using the time-averaged control,
we obtain the required stabilizing effect. We consider the
following initial density of particles centred at a
point $z_0=(x_0,v_0)$ that does not belong to the desired trajectory:
\begin{align*}
f_0 = \mathcal{N}_2\left(z_0, \Sigma_0 \right),
\qquad z_0 = \binom{8.0}{3.5},
\qquad \Sigma_0 = \begin{pmatrix}
	\textcolor{black}{0.15}& 0 \\ 0 & \textcolor{black}{0.15}
\end{pmatrix} .
\end{align*}

Then, we simulate the evolution of these particles subject to the
time-averaged control $\bar{u}$ given by \eqref{eControlStat} with
$u$ computed in the previous experiment (see \cref{fig:Control_quiver}).
The stationary control field $\bar{u}$ is depicted in \cref{fig:TimeAveraged_Control}.

\begin{figure}[h]
\centering
\includegraphics[width=0.4\textwidth]{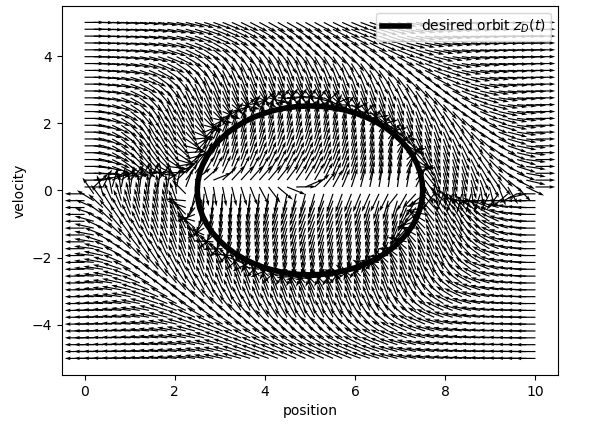}
\hfill
\includegraphics[width=0.4\textwidth]{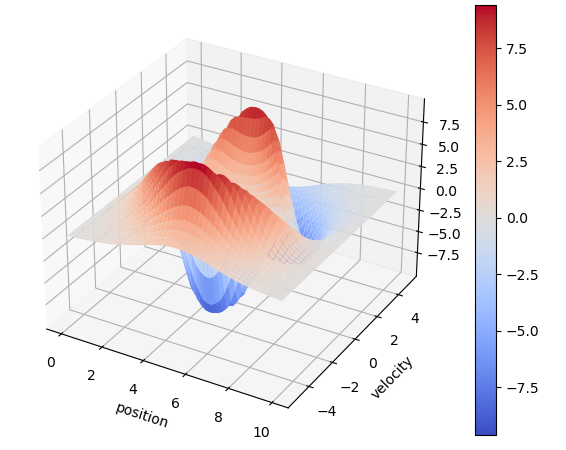}
\caption{Averaged control $\bar{u}$ defined in \eqref{eControlStat}; quiver and 3D plot.}
\label{fig:TimeAveraged_Control}
\end{figure}

In \cref{fig:control_test_cases}, we see that the initial density of particles
is far away from the desired orbit. Nevertheless, subject to
the stationary control field, we see that the particles are driven towards
this orbit.

\begin{figure}
\includegraphics[width=\scaleParticlePlots\textwidth]{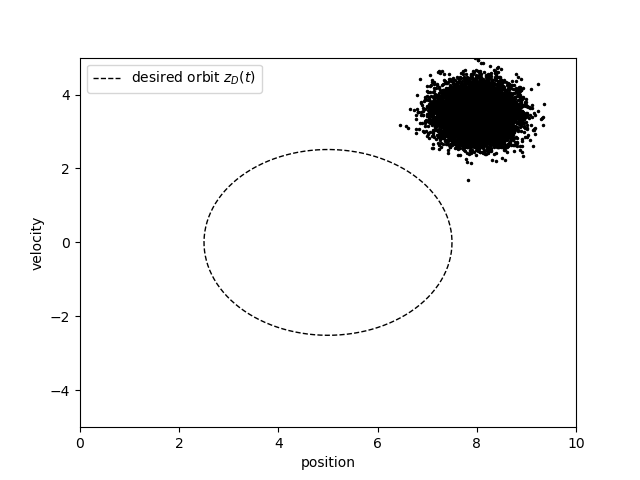}
\includegraphics[width=\scaleParticlePlots\textwidth]{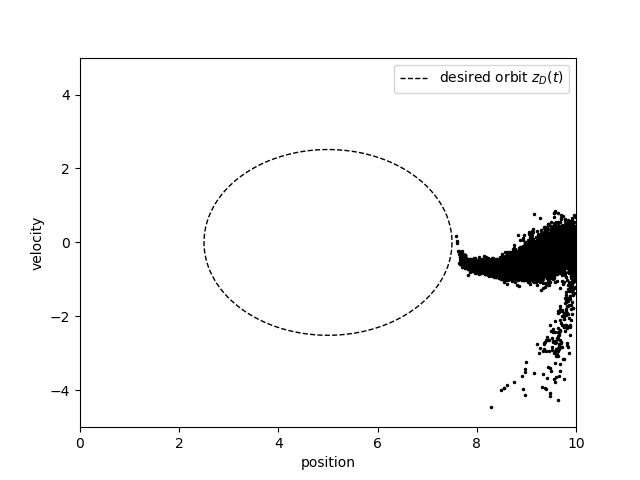}
\includegraphics[width=\scaleParticlePlots\textwidth]{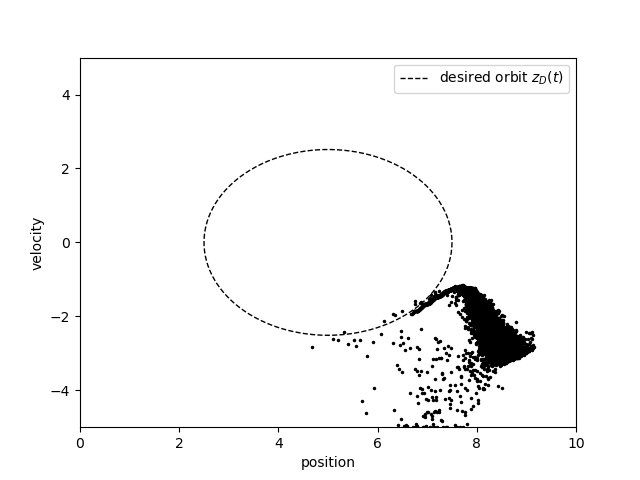}
\includegraphics[width=\scaleParticlePlots\textwidth]{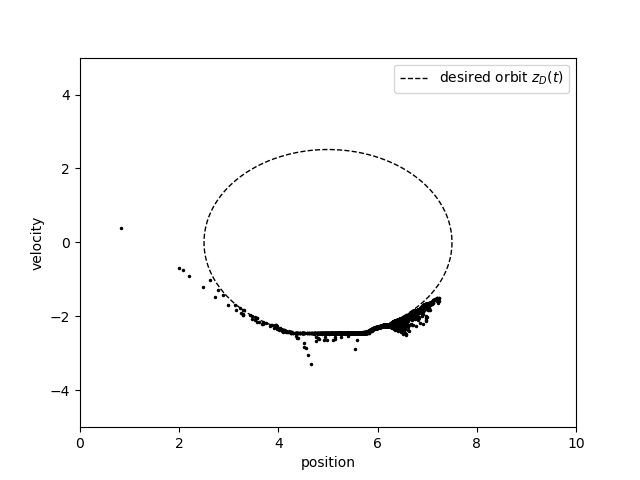}
\includegraphics[width=\scaleParticlePlots\textwidth]{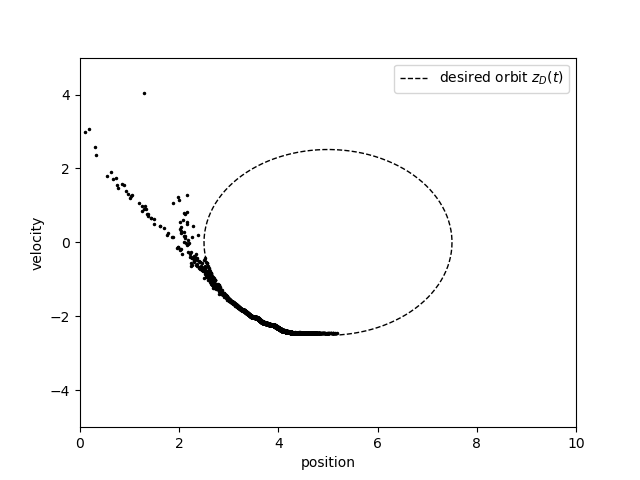}
\includegraphics[width=\scaleParticlePlots\textwidth]{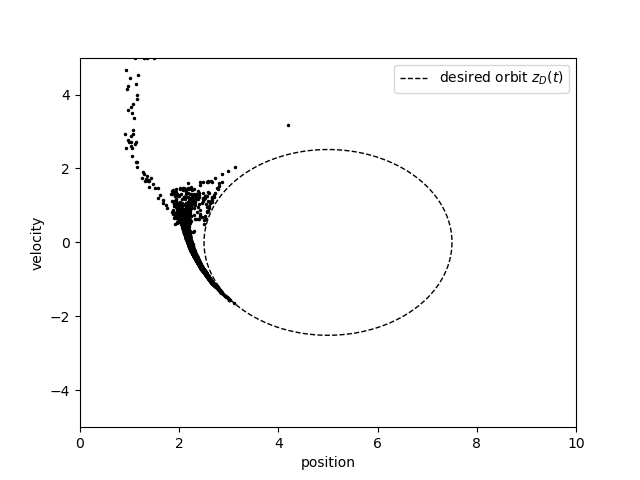}
\includegraphics[width=\scaleParticlePlots\textwidth]{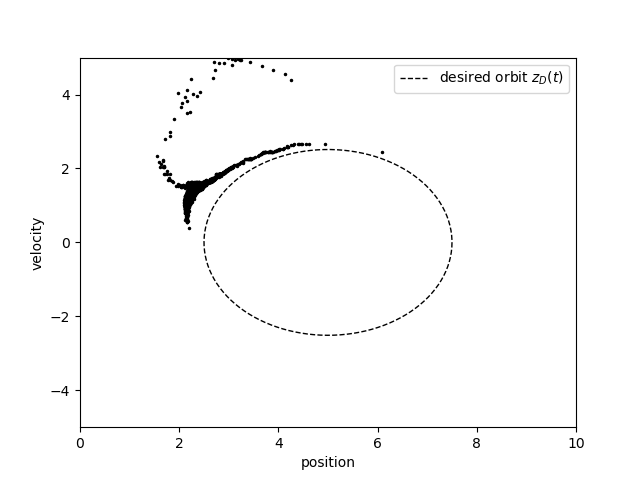}
\hfill
\includegraphics[width=\scaleParticlePlots\textwidth]{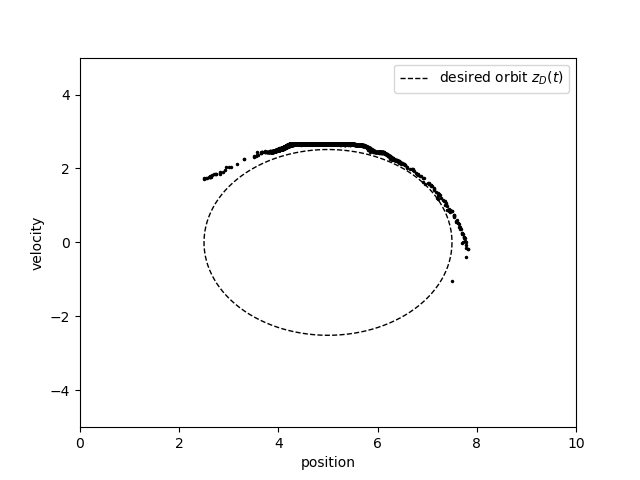}
\caption{{Evolution of $f$, starting with an initial Gaussian distribution
		and subject to the averaged control $\bar{u}$. Time ordering as in \cref{fig:Control_quiver}}.}
\label{fig:control_test_cases}
\end{figure}

\clearpage

\section{Conclusion}
\label{sec:Conclusion}

This work was devoted to the construction
of feedback-like control fields for a kinetic model in phase space. The purpose of these controls was to drive any initial density of
particles in the phase space to reach and maintain in a stable way a desired
orbit. For this purpose, a one-shot method was presented that
was based on the formulation of an ensemble optimal control problem governed by the governing kinetic model. The one-shot solution procedure
consisted of a Monte Carlo backward-in-time solve of a nonlinear
augmented adjoint kinetic model. Results of numerical
experiments were presented that demonstrated the effectiveness of the proposed feedback control framework.

\section*{Acknowledgments}
The author J.B. was partially financially supported by the SFB1432.
We would like to express our gratitude to the anonymous referees for their helpful questions and remarks.


\begin{thebibliography}{10}
	
	\bibitem{AlbiChoiFornasierKalise2017MeanFieldControlHierarchy}
	{\sc Albi, G., Choi, Y.-P., Fornasier, M., and Kalise, D.}
	\newblock Mean field control hierarchy.
	\newblock {\em Appl. Math. Optim. 76}, 1 (2017), 93--135.
	
	\bibitem{Bartsch2021MOCOKI}
	{\sc Bartsch, J., and Borz\`\i, A.}
	\newblock M{OCOKI}: {A} {M}onte {C}arlo approach for optimal control in the
	force of a linear kinetic model.
	\newblock {\em Comput. Phys. Commun. 266\/} (2021), 108030.
	
	\bibitem{Bartsch2019Theoretical}
	{\sc Bartsch, J., Borz\`\i, A., Fanelli, F., and Roy, S.}
	\newblock A theoretical investigation of {B}rockett's ensemble optimal control
	problems.
	\newblock {\em Calc. Var. Partial Differential Equations 58}, 5 (2019), Paper
	No. 162.
	
	\bibitem{Bartsch2020OCPKS}
	{\sc Bartsch, J., Nastasi, G., and Borz\`\i, A.}
	\newblock Optimal control of the {K}eilson-{S}torer master equation in a
	{M}onte {C}arlo framework.
	\newblock {\em J. Comput. Theor. Transp.\/} (2021).
	
	\bibitem{BealsProtopopescu1987AbstractTransport}
	{\sc Beals, R., and Protopopescu, V.}
	\newblock Abstract time-dependent transport equations.
	\newblock {\em J. Math. Anal. Appl. 121}, 2 (1987), 370--405.
	
	\bibitem{Bellman1957}
	{\sc Bellman, R.}
	\newblock {\em Dynamic Programming}.
	\newblock Princeton University Press, 1957.
	
	\bibitem{Berman2010}
	{\sc Berman, P., and Malinovsky, V.}
	\newblock {\em Principles of Laser Spectroscopy and Quantum Optics}.
	\newblock Princeton University Press, 2010.
	
	\bibitem{Berman1986}
	{\sc Berman, P.~R., Haverkort, J. E.~M., and Woerdman, J.~P.}
	\newblock Collision kernels and transport coefficients.
	\newblock {\em Phys. Rev. A 34\/} (Dec 1986), 4647--4656.
	
	\bibitem{Box1958}
	{\sc Box, G. E.~P., and Muller, M.~E.}
	\newblock A note on the generation of random normal deviates.
	\newblock {\em Ann. Math. Statist. 29}, 2 (1958), 610--611.
	
	\bibitem{Brockett1997}
	{\sc Brockett, R.~W.}
	\newblock Minimum attention control.
	\newblock In {\em Proceedings of the 36th IEEE Conference on Decision and
		Control\/} (1997), vol.~3, IEEE, pp.~2628--2632.
	
	\bibitem{Brockett2007}
	{\sc Brockett, R.~W.}
	\newblock Optimal control of the {L}iouville equation.
	\newblock In {\em Proceedings of the {I}nternational {C}onference on {C}omplex
		{G}eometry and {R}elated {F}ields\/} (2007), vol.~39 of {\em AMS/IP Stud.
		Adv. Math.}, Amer. Math. Soc., Providence, RI, pp.~23--35.
	
	\bibitem{Brockett2012}
	{\sc Brockett, R.~W.}
	\newblock Notes on the control of the {L}iouville equation.
	\newblock In {\em Control of partial differential equations}, vol.~2048 of {\em
		Lecture Notes in Math.} Springer, Heidelberg, 2012, pp.~101--129.
	
	\bibitem{CaflischSilantyevYang2021AdjointDSMC}
	{\sc Caflisch, R., Silantyev, D., and Yang, Y.}
	\newblock Adjoint {DSMC} for nonlinear {B}oltzmann equation constrained
	optimization.
	\newblock {\em J. Comput. Phys. 439\/} (2021), Paper No. 110404, 29.
	
	\bibitem{ChenYang1991LinearTransportSpecularReflection}
	{\sc Chen, J., and Yang, M.~Z.}
	\newblock Linear transport equation with specular reflection boundary
	condition.
	\newblock {\em Transport Theory Statist. Phys. 20}, 4 (1991), 281--306.
	
	\bibitem{Fabbri2017}
	{\sc Fabbri, G., Gozzi, F., and {\'S}wi{\k{e}}ch, A.}
	\newblock {\em Stochastic Optimal Control in Infinite Dimension: Dynamic
		Programming and HJB Equations}.
	\newblock Probability Theory and Stochastic Modelling. Springer International
	Publishing, 2017.
	
	\bibitem{Gelin2015}
	{\sc Gelin, M.~F., Blokhin, A.~P., Tolkachev, V.~A., and Domcke, W.}
	\newblock Microscopic derivation of the {K}eilson – {S}torer master equation.
	\newblock {\em J. Chem. Phys. 462\/} (2015), 35 -- 40.
	
	\bibitem{Gelin2006}
	{\sc Gelin, M.~F., and Kosov, D.~S.}
	\newblock Molecular reorientation in hydrogen-bonding liquids: Through
	algebraic t - 3/2 relaxation toward exponential decay.
	\newblock {\em J. Chem. Phys. 124}, 14 (2006), 144514.
	
	\bibitem{Hairer03}
	{\sc Hairer, E., Lubich, C., and Wanner, G.}
	\newblock Geometric numerical integration illustrated by the
	{S}t{\"o}rmer-{V}erlet method.
	\newblock {\em Acta Numer. 12\/} (2003), 399--450.
	
	\bibitem{HorowitzDamleBurdick2014LinearHJB}
	{\sc Horowitz, M.~B., Damle, A., and Burdick, J.~W.}
	\newblock Linear hamilton jacobi bellman equations in high dimensions.
	\newblock In {\em 53rd IEEE Conference on Decision and Control\/} (2014), IEEE,
	pp.~5880--5887.
	
	\bibitem{Jacoboni1983}
	{\sc Jacoboni, C., and Reggiani, L.}
	\newblock The {M}onte {C}arlo method for the solution of charge transport in
	semiconductors with applications to covalent materials.
	\newblock {\em Rev. Mod. Phys. 55}, 3 (1983), 645--705.
	
	\bibitem{KeilsonStorer52}
	{\sc Keilson, J., and Storer, J.~E.}
	\newblock On {B}rownian motion, {B}oltzmann's equation, and the
	{F}okker-{P}lanck equation.
	\newblock {\em Quart. Appl. Math. 10\/} (1952), 243--253.
	
	\bibitem{Kosov2018}
	{\sc Kosov, D.~S.}
	\newblock Telegraph noise in {M}arkovian master equation for electron transport
	through molecular junctions.
	\newblock {\em J. Chem. Phys. 148}, 18 (2018).
	
	\bibitem{LakshmiParvathy12}
	{\sc Lakshmi, K., Parvathy, R., Soumya, S., and Soman, K.}
	\newblock Image denoising solutions using heat diffusion equation.
	\newblock In {\em 2012 International Conference on Power, Signals, Controls and
		Computation\/} (2012), IEEE, pp.~1--5.
	
	\bibitem{LatrachLods2009TransportBounceBack}
	{\sc Latrach, K., and Lods, B.}
	\newblock Spectral analysis of transport equations with bounce-back boundary
	conditions.
	\newblock {\em Math. Methods Appl. Sci. 32}, 11 (2009), 1325--1344.
	
	\bibitem{LiWangYang2023MCGradient}
	{\sc Li, Q., Wang, L., and Yang, Y.}
	\newblock Monte {C}arlo {G}radient in {O}ptimization {C}onstrained by
	{R}adiative {T}ransport {E}quation.
	\newblock {\em SIAM J. Numer. Anal. 61}, 6 (2023), 2744--2774.
	
	\bibitem{Skeel97}
	{\sc Skeel, R.~D., Zhang, G., and Schlick, T.}
	\newblock A family of symplectic integrators: stability, accuracy, and
	molecular dynamics applications.
	\newblock {\em SIAM J. Sci. Comput. 18\/} (1997), 203--222.
	
	\bibitem{SpeyerEvans1984}
	{\sc Speyer, J., and Evans, R.}
	\newblock A second variational theory for optimal periodic processes.
	\newblock {\em IEEE Trans. Automat. Contr. 29}, 2 (1984), 138--148.
	
	\bibitem{Strekalov2012}
	{\sc Strekalov, M.~L.}
	\newblock Population relaxation of highly rotationally excited molecules at
	collisions.
	\newblock {\em Chem. Phys. Lett. 548\/} (2012), 7 -- 11.
	
	\bibitem{Tran2009}
	{\sc Tran, H., Hartmann, J.-M., Chaussard, F., and Gupta, M.}
	\newblock An isolated line-shape model based on the {K}eilson-{S}torer function
	for velocity changes. ii. molecular dynamics simulations and the q(1) lines
	for pure h2.
	\newblock {\em J. Chem. Phys. 131}, 15 (2009), 154303.
	
	\bibitem{VanderMee1991NonDivFreeForceKinEQ}
	{\sc van~der Mee, C. V.~M.}
	\newblock Trace theorems and kinetic equations for non-divergence-free external
	forces.
	\newblock {\em Appl. Anal. 41}, 1-4 (1991), 89--110.
	
	\bibitem{Verlet67}
	{\sc Verlet, L.}
	\newblock Computer ``experiments" on classical fluids. {I}. {T}hermodynamical
	properties of {L}ennard-{J}ones molecules.
	\newblock {\em Physical review 159}, 1 (1967), 98.
	
	\bibitem{YangSilantyevCaflisch2023AdjointDSMC_GeneralCollisionModel}
	{\sc Yang, Y., Silantyev, D., and Caflisch, R.}
	\newblock Adjoint {DSMC} for nonlinear spatially-homogeneous {B}oltzmann
	equation with a general collision model.
	\newblock {\em J. Comput. Phys. 488\/} (2023), Paper No. 112247.
	
\end{thebibliography}
\end{document}